\documentclass[11pt, a4paper]{article}

\usepackage[utf8]{inputenc}
\usepackage[english]{babel}
\usepackage[T1]{fontenc}
\usepackage[colorlinks = true,
            linkcolor = blue,
            urlcolor  = blue,
            citecolor = blue,
            anchorcolor = blue,
            hyperfootnotes=false]{hyperref}
\usepackage[a4paper,width=150mm,top=25mm,bottom=25mm]{geometry}
\usepackage{fancyhdr}
\usepackage[symbol]{footmisc}

\usepackage{amsmath,amsthm,amssymb, enumitem, mathtools, mathrsfs}
\usepackage{siunitx}
\setitemize{noitemsep}

\usepackage{tikz}
\usepackage{tikz-cd}
\usepackage{caption}
\usepackage[all]{xy}
\usepackage{graphicx}
\usepackage{chngcntr}
\usepackage{booktabs}
\usepackage[labelformat=simple]{subcaption}

\usepackage{csquotes}
\usepackage[backend=biber,
            sorting=anyt,
            maxbibnames=4]{biblatex}
\graphicspath{ {images/} }

\theoremstyle{plain}
\newtheorem{theorem}{Theorem}[section]
\newtheorem*{theorem*}{Theorem}
\newtheorem{proposition}[theorem]{Proposition}
\newtheorem{lemma}[theorem]{Lemma}

\theoremstyle{definition}
\newtheorem{definition}[theorem]{Definition}
\newtheorem{assumption}[theorem]{Assumption}

\theoremstyle{remark}
\newtheorem{remark}[theorem]{Remark}

\numberwithin{equation}{section}

\DeclareMathOperator{\id}{id}

\DeclareMathOperator{\sgn}{sgn}

\newcommand{\ev}{\mathbb{E}}
\newcommand{\pr}{\mathbb{P}}

\newcommand{\R}{\mathbb{R}}
\renewcommand{\P}{\mathcal{P}}
\newcommand{\F}{\mathcal{F}}
\renewcommand{\L}{\mathcal{L}}

\renewcommand{\d}{\mathrm{d}}
\newcommand{\define}{\mathpunct{:}}

\newcommand{\bb}[1]{\mathbb{#1}}
\renewcommand{\bf}[1]{\mathbf{#1}}
\renewcommand{\cal}[1]{\mathcal{#1}}

\begin{document}

\noindent
\begin{center}
    \Large
    \textbf{The Atlas Model and SDEs with Boundary Interaction}

    \vspace{1em}

    \normalsize
    Philipp Jettkant\footnote[1]{Department of Mathematics, Imperial College London, UK, \href{mailto:p.jettkant@imperial.ac.uk}{p.jettkant@imperial.ac.uk}.} 

\end{center}

\vspace{1em}

\begin{abstract}
We study the mean-field limit of the Atlas model and its connection to SDEs with dependence on the distribution of hitting and local times. The Atlas model describes a system of Brownian particles on the real line, where only the lowest ranked particle receives a positive drift, proportional to the number of particles. We show that in the mean-field limit the particle system converges to a novel SDE with reflection at a moving boundary, whose motion is such that the average local time spent at the boundary grows at a constant rate. In general, the boundary is represented by a measure, so the reflection must be interpreted in a relaxed sense. However, for sufficiently regular initial particle profiles, we prove that the boundary is a continuous function. Our analysis relies on a reformulation of the problem via McKean--Vlasov SDEs with interaction through hitting and local times.
\end{abstract}

% \section*{Open Questions}

% Here are some open questions.

% \begin{itemize}
%     \item Can the link between the reflected SDE \eqref{eq:mfl} with growth constraint on the expected local time and McKean--Vlasov SDEs \eqref{eq:probab_repr_hitting} and \eqref{eq:probab_repr_reflected} be made directly, without the need to appeal to the Fokker--Planck equation \eqref{eq:fpe}?
%     \item In the case that the Atlas model \eqref{eq:ps} converges to a strong solution $b$ of \eqref{eq:mfl}, does the Atlas stock $B^N$ tend to $b$ uniformly on compacts in probability?
%     \item What can be said regarding McKean--Vlasov SDEs \eqref{eq:probab_repr_hitting} and \eqref{eq:probab_repr_reflected} in the supercritical case?
%     \item Can results be extended to generalised Atlas models from \cite{banner_atlas_2005}?
% \end{itemize}

\section{Introduction}

In this article, we study the mean-field limit of a system of Brownian particles interacting through rank-dependent drifts. We consider the Atlas model, in which $N$ particles diffuse on the real line according to independent Brownian motions, with the lowest ranked particle, called the Atlas particle, receiving a drift with intensity $\gamma N$ for $\gamma > 0$. Formally, the state $X^i = (X^i_t)_{t \geq 0}$ of particle $i \in \{1, \dots, N\}$ evolves according to the dynamics
\begin{equation} \label{eq:ps}
    \d X^i_t = \gamma N \bf{1}_{\{X^i_t = X^{(1)}_t\}} \, \d t + \d W^i_t,
\end{equation}
starting from the initial condition $X^i_0 = \xi_i$. The processes $W^i = (W^i_t)_{t \geq 0}$, $i \in \{1, \dots, N\}$, are independent Brownian motions and $X^{(1)}_t$,~\ldots, $X^{(N)}_t$ denote the order statistics or ranks of the states $X^1_t$,~\ldots, $X^N_t$. That is,
\begin{equation*}
    X^{(1)}_t \leq X^{(2)}_t \leq \dots \leq X^{(N)}_t.
\end{equation*}
Throughout the article, we assume that the support of the common law of all the initial conditions $\xi_1$,~\ldots, $\xi_N$ is a subset of $[0, \infty)$ that includes the origin. 

The Atlas model \eqref{eq:ps} was introduced by Fernholz in his seminal monograph \cite{fernholz_spt_2002} on Stochastic Portfolio Theory (see Example 5.3.3). In this context, the particle states represent the logarithmic market capitalisations of constituents of a large stock market. Since then, there has been a steady interest in finite and infinite versions of the Atlas model \cite{banner_atlas_2005, pal_ranked_bm_2008, chatterjee_rank_2010, pal_ranked_concentration_2014,  dembo_atlas_2017, cabezas_rank_2019, atar_atlas_2025} and more general systems with coefficients depending on particle ranks. Our goal is to understand the behaviour of the system \eqref{eq:ps} as the number $N$ of particles is taken to infinity. In particular, we wish to derive an SDE that describes the dynamics of a representative particle in the limit $N \to \infty$. Due to the interactive nature of the system, we should expect this SDE to be of McKean--Vlasov type.

Unlike for rank-based particle systems whose coefficients do not scale with the size $N$ of the system, the structure of the Atlas model \eqref{eq:ps} does not immediately suggest an obvious candidate for the mean-field limit. The main impediment to deriving such a limit is the singular nature of the drift provided to the Atlas particle. However, the matter becomes clearer upon reformulation of \eqref{eq:ps} as a system of reflected diffusions. Let us denote the state $X^{(1)}$ of the Atlas particle by $B^N = (B^N_t)_{t \geq 0}$ and define the nondecreasing processes $L^i = (L^i_t)_{t \geq 0}$, $i = 1$,~\ldots, $N$, by
\begin{equation*}
    L^i_t = \int_0^t \gamma N \bf{1}_{\{X^i_s = B^N_s\}} \, \d s
\end{equation*}
for $t \geq 0$. Note that $\frac{1}{N} \sum_{j = 1}^N L^j_t = \gamma t$. Now, we may view the Atlas model \eqref{eq:ps} as a system of SDEs
\begin{equation} \label{eq:ps_reflected}
    \d X^i_t = \d W^i_t + \d L^i_t
\end{equation}
with reflection at the barrier $B^N$, determined by the constraint $\frac{1}{N} \sum_{j = 1}^N L^j_t = \gamma t$. Indeed, it is easy to see that $(X^i - B^N, L^i)$ solves the Skorokhod problem for the process $\xi_i + W^i - B^N$ for $i = 1$,~\ldots, $N$, since $X^i_t - B^N_t = \xi_i + W^i_t - B^N_t + L^i_t \geq 0$, $L^i$ is continuous, nondecreasing, and started from zero, and
\begin{equation*}
    \int_0^t (X^i_s - B^N_s) \, \d L^i_t = \gamma N \int_0^t (X^i_s - B^N_s) \bf{1}_{\{X^i_s = B^N_s\}} \, \d s = 0
\end{equation*}
for $t \geq 0$. The defining features of the Skorokhod problem in one dimension are recalled in Appendix \ref{sec:skorokhod_problem}, Definition \ref{def:skorokhod_problem}. 

\subsection{The Mean-Field Limit}

An obvious candidate now presents itself as the limit for the reflected SDE \eqref{eq:ps_reflected}. Indeed, naively taking the limit as $N \to \infty$ in the constraint for the local times $L^1$,~\ldots, $L^N$, suggests that the state $X = (X_t)_{t \geq 0}$ of the representative particle in the mean-field limit should satisfy the SDE
\begin{equation} \label{eq:mfl}
    \d X_t =  \d W_t + \d L_t
\end{equation}
with initial condition $X_0 = \xi \sim \xi_1$, reflected at a deterministic barrier $b = (b_t)_{t \geq 0}$ such that $\ev[L_t] = \gamma t$. Since the constraint is on the law of the local time of the representative particle, we will refer to \eqref{eq:mfl} as a reflected McKean--Vlasov SDE. We understand this equation in the following way.

\begin{definition} \label{def:solution_strong}
We say that a continuous function $b \define [0, \infty) \to \R$ is a solution of McKean--Vlasov SDE \eqref{eq:mfl} if the solution $(\tilde{X}, L)$ of the Skorokhod problem for $\xi + W - b$ satisfies $\ev[L_t] = \gamma t$ for $t \geq 0$.
\end{definition}

Surprisingly, it is straightforward to show that if it exists, the solution to McKean--Vlasov SDE \eqref{eq:mfl} is unique.

\begin{proposition} \label{prop:mfl_strong_uniqueness}
McKean--Vlasov SDE \eqref{eq:mfl} has at most one solution.
\end{proposition}

\begin{proof}
Let $b^1$ and $b^2$ be solutions to McKean--Vlasov SDE \eqref{eq:mfl}. Denote the solutions to the Skorokhod problem for $\xi + W - b^i$ by $(X^i - b^i, L^i)$, $i = 1$, $2$. Then, we have by the fundamental theorem of calculus that
\begin{align*}
    \lvert X^1_t - X^2_t\rvert^2 &= \int_0^t (X^1_s  - X^2_s) \, \d L^1_s - \int_0^t (X^1_s  - X^2_s) \, \d L^2_s \\
    &\leq \int_0^t (b^1_s  - b^2_s) \, \d L^1_s - \int_0^t (b^1_s  - b^2_s) \, \d L^2_s \\
    &= \int_0^t (b^1_s  - b^2_s) \, \d (L^1_s - L^2_s).
\end{align*}
Here we used in the second line the properties of the Skorokhod problem, namely that $\int_0^t X^i_s \, \d L^i_s = \int_0^t b^i_s \, \d L^i_s$ for $i \in \{1, 2\}$ and $\int_0^t X^i_s \, \d L^j_s \geq \int_0^t b^i_s \, \d L^j_s$ for $i$, $j \in \{1, 2\}$. Taking expectation on both sides of the above inequality implies that
\begin{equation*}
    \ev\bigl[\lvert X^1_t - X^2_t\rvert^2\bigr] \leq \ev\biggl[\int_0^t (b^1_s  - b^2_s) \, \d (L^1_s - L^2_s)\biggr] = 0,
\end{equation*}
since $\ev[L^1_t] = \ev[L^2_t] = \gamma t$. Due to the continuity of $X^1$ and $X^2$, it follows that a.s.\@ $X^1_t = X^2_t$ for all $t \geq 0$. This implies that a.s.\@ $L^1_t = L^2_t$ for $t \geq 0$, from which we deduce that
\begin{equation*}
    \int_0^t b^1_s \, \d L^1_s = \int_0^t X^1_s \, \d L^1_s = \int_0^t X^2_s \, \d L^2_s = \int_0^t b^2_s \, \d L^2_s.
\end{equation*}
We take expectations on both sides of this equality and divide by $\gamma$ to find $\int_0^t b^1_s \, \d s = \int_0^t b^2_s \, \d s$ for $t \geq 0$. Since $b^1$ and $b^2$ are continuous, it follows that they coincide. This concludes the proof.
\end{proof}

% as saying that $(X - b, L)$ is a solution to the Skorokhod problem for $\xi + W - b$ and $\ev[L_t] = \gamma t$. For the former to be well-posed, we must require that the path $b$ is c\`adl\`ag. Then the solution property of the Skorokhod problem asks that $X_t - b_t = \xi_t + W_t - b_t + L_t \geq 0$, that $L$ is c\`adl\`ag, nondecreasing, and started from zero, and that $\int_0^t (X_s - b_s) \, \d L_s = 0$ for $t \geq 0$. 
Is it reasonable, however, to suppose that the limit of the barrier $B^N$ in the finite system satisfies McKean--Vlasov \eqref{eq:mfl} in the sense of Definition \ref{def:solution_strong}? Note that the barrier $B^N$ is the Atlas particle whose dynamics are given by
\begin{equation} \label{eq:atlas_particle}
    \d B^N_t = \gamma N \, \d t + \d \tilde{W}_t - \frac{1}{2} \, \d L^{X^{(2)} - X^{(1)}}_t,
\end{equation}
where $\tilde{W} = (\tilde{W}_t)_{t \geq 0}$ is a Brownian motion and $L^{X^{(2)} - X^{(1)}} = (L^{X^{(2)} - X^{(1)}}_t)_{t \geq 0}$ denotes the local time at zero of the first gap $X^{(2)} - X^{(1)}$. This formula follows e.g.\@ from \cite[Proposition 4.1.11]{fernholz_spt_2002}, which derives the dynamics of ranked processes. A priori, it is not at all clear that the exploding drift $\gamma N t$ on the right-hand side of \eqref{eq:atlas_particle} should be offset by the local time at zero of the gap $X^{(2)} - X^{(1)}$ in precisely such a way that we wind up with a continuous and deterministic barrier $b$. This motivates us to search for a relaxation of the reflected McKean--Vlasov SDE \eqref{eq:mfl}, which make sense even when the limit of $(B^N)_{N \geq 1}$ could in principle be a much rougher object.

To that end, let us introduce the random measure $\beta^N$ on $[0, \infty) \times \R$ given by
\begin{equation*}
    \d \beta^N(t, x) = \d \delta_{B^N_t}(x) \d t.
\end{equation*}
Then, as soon as we can get sufficient control on the magnitude of $B^N$ as $N \to \infty$, it follows that $(\beta^N)_{N \geq 1}$ is tight on $\cal{M}_1([0, \infty) \times \R)$. Here $\cal{M}_1([0, \infty) \times \R)$ is the set of measures $m$ on $[0, \infty) \times \R$ such that $m([0, t] \times \R) = t$ for $t \geq 0$, topologised with the coarsest topology that makes the maps $\cal{M}_1([0, \infty) \times \R) \to \P([0, n] \times \R)$, $m \mapsto \frac{1}{n} m \vert_{[0, n] \times \R}$ for $n \in \bb{N}$, continuous. If the sequence $(\beta^N)_{N \geq 1}$ is tight on $\cal{M}_1([0, \infty) \times \R)$, it admits a subsequential weak limit $\beta$ with values in $\cal{M}_1([0, \infty) \times \R)$. We wish to phrase our generalisation of McKean--Vlasov SDE \eqref{eq:mfl} in terms of such $\beta$. As before, the state $(X_t)_{t \geq 0}$ of the representative particle takes the form $X_t = \xi + W_t + L_t$ for a nondecreasing continuous process $L = (L_t)_{t \geq 0}$ started from zero. However, we can no longer conceive of $L$ as a regulator process arising from reflection at a barrier. In particular, we have to weaken the conditions $X_t \geq b_t$ and $\int_0^t (X_s - b_s) \, \d L_s = 0$, characterising the Skorokhod problem underlying the strong formulation of McKean--Vlasov SDE \eqref{eq:mfl}. 

This is addressed by the following definition, for which we fix a filtered probability space $(\Omega, \F, \bb{F}, \pr)$ carrying an $\F_0$-measurable random variable $\xi \sim \xi_1$, an $\bb{F}$-Brownian motion $W$, as well as a subfiltration $\bb{G}$ of $\bb{F}$ independent of $(\xi, W)$. We also denote by $\bb{F}^{\xi, W} = (\F^{\xi, W}_t)_{t \geq 0}$ the filtration given by $\F^{\xi, W}_t = \sigma(\xi, W_s \define s \in [0, t])$ for $t \geq 0$ and say that an $\cal{M}_1([0, \infty) \times \R)$-valued random variable is $\bb{G}$-adapted if for all $t \geq 0$, the random variable $\beta([0, s] \times A)$ is $\cal{G}_t$-measurable for $s \in [0, t]$ and $A \in \cal{B}(\R)$.

\begin{definition} \label{def:generalised_solution}
We say that a tuple $(L, \beta)$ consisting of an integrable nondecreasing continuous $\bb{F}$-adapted stochastic process $L$ and an $\cal{M}_1([0, \infty) \times \R)$-valued $\bb{G}$-adapted random measure $\beta$ is a generalised solution of McKean--Vlasov SDE \eqref{eq:mfl} if 
\begin{enumerate}[noitemsep, label = (\roman*)]
    \item \label{it:reflection} for all $\varphi \in C_b([0, \infty) \times \R)$ and all $T > 0$, we have a.s.\@
    \begin{equation}
        \ev\biggl[\int_0^T \varphi(t, X_t) \, \d L_t \bigg\vert \cal{G}_T \biggr] = \gamma \int_{[0, T] \times \R} \varphi(t, x) \, \d \beta(t, x);
    \end{equation}
    \item \label{it:minimality} for all $\varphi \in C_b([0, \infty) \times \R)$ such that $x \mapsto \varphi(t, x)$ is nondecreasing for $t \geq 0$, all integrable nondecreasing continuous $\bb{F}^{\xi, W}$-adapted stochastic processes $\tilde{L}$ with $\ev[\tilde{L}_t] = \gamma t$ for $t \geq 0$, and all $T > 0$, we have a.s.\@
    \begin{equation}
        \ev\biggl[\int_0^T \varphi(t, X_t) \, \d \tilde{L}_t \bigg\vert \cal{G}_T \biggr] \geq \gamma\int_{[0, T] \times \R} \varphi(t, x) \, \d \beta(t, x).
    \end{equation}
\end{enumerate}
Here $X = \xi + W + L$.
\end{definition}

Condition \ref{it:reflection} encodes both the fact that $\ev[L_t] = \gamma t$ and that when the state is located at the boundary, its law is given by $\beta$. Condition \ref{it:minimality} captures that $X_t$ should stay above the boundary. Note that if $b$ is a solution to McKean--Vlasov SDE \eqref{eq:mfl} in the strong sense, then if $(X - b, L)$ is the solution to the Skorokhod problem for $\xi + W - b$, the couple $(L, \beta)$, where $\d \beta(t, x) = \d \delta_{b_t}(x) \, \d t$, solves McKean--Vlasov SDE \eqref{eq:mfl} in the sense of Definition \ref{def:generalised_solution}. Hence, Definition \ref{def:generalised_solution} indeed generalises Definition \ref{def:solution_strong}.

As hinted at earlier, it can be shown that $(\beta^N)_{N \geq 1}$ is tight on $\cal{M}_1([0, \infty) \times \R)$. Moreover, we will prove that for any $i \in \{1, \dots, N\}$, the sequence $(L^{N, i})_{N \geq i}$ is tight on $C([0, \infty))$. Note that we add the subscript $N$ to the process $L^{N, i}$ for emphasis. Any limit point of the two sequences is a solution to McKean--Vlasov SDE \eqref{eq:mfl} in the sense of Definition \ref{def:generalised_solution}.

\begin{theorem} \label{thm:convergence}
For any $i \in \bb{N}$, the sequence $(L^{N, i}, \beta^N)_{N \geq i}$ is tight on the space $C([0, \infty)) \times \cal{M}_1([0, \infty) \times \R)$ and any weak limit point is a solution to McKean--Vlasov SDE \eqref{eq:mfl} in the sense of Definition \ref{def:generalised_solution}.
\end{theorem}

The major concern with a generalised notion of solution is that it may be so weak that it admits too many candidates as solutions. The most definite way to dispel this worry is to establish uniqueness. This can indeed be achieved.

\begin{theorem} \label{thm:unique_generalised}
McKean--Vlasov SDE \eqref{eq:mfl}, understood in the sense of Definition \ref{def:generalised_solution}, exhibits pathwise uniqueness. For the unique solution $(L, \beta)$ it holds that $L$ is $\bb{F}^{\xi, W}$-adapted and $\beta$ is deterministic. In particular, for any $i \in \bb{N}$, the sequence $(L^{N, i}, \beta^N)_{N \geq i}$ converges weakly to $(L, \beta)$.
\end{theorem}

The proof of Theorem \ref{thm:unique_generalised} proceeds via an approximation argument, which exploits that for initial conditions with a sufficiently regular distribution, McKean--Vlasov SDE \eqref{eq:mfl} can in fact be solved in the strong sense. A generalised solution to McKean--Vlasov SDE \eqref{eq:mfl} for an arbitrary initial condition is then uniquely determined as the limit of an approximating sequence of strong solutions with regular initial conditions. 

The most challenging step in this argument consists in establishing the existence of strong solutions for regular initial conditions. Our construction of such strong solutions relies on an alternative representation of McKean--Vlasov SDE \eqref{eq:mfl} through another McKean--Vlasov SDE with reflection. In the course of the derivation of this representation, we shall encounter the so-called supercooled Stefan problem. Before proceeding to this, let us state the existence result for strong solutions of McKean--Vlasov SDE \eqref{eq:mfl}, starting with suitable assumptions.

\begin{assumption} \label{ass:bounded_variation}
We assume that the distribution of $\xi$ has a density given by a c\`adl\`ag function $[0, \infty) \to [0, \infty)$ with finite total variation that does not vanish at zero.
\end{assumption}

A variation of Assumption \ref{ass:bounded_variation} naturally arises in the treatment of the related supercooled Stefan problem by Delarue, Nadtochiy \& Shkolnikov \cite{delarue_stefan_2022}. We shall see further below how it enters our analysis. Under Assumption \ref{ass:bounded_variation} we have the following existence result.

\begin{theorem} \label{thm:mfl_exist}
Let Assumption \ref{ass:bounded_variation} be satisfied. Then McKean--Vlasov SDE \eqref{eq:mfl} has a unique solution in the sense of Definition \ref{def:solution_strong}.
\end{theorem}

\begin{remark}
If $\xi$ follows an exponential distribution with rate $2\gamma$, then the solution $b$ to McKean--Vlasov SDE \eqref{eq:mfl} is given by $b_t = \gamma t$ for $t \geq 0$. In that case, the distribution $\L(\tilde{X}_t)$ of the solution $(\tilde{X}, L)$ to the Skorokhod problem for $\xi + W - b$ is stationary. Moreover, one easily verifies that $\textup{Exp}(2\gamma)$ is the only initial distribution for which this is true. Indeed, stationarity implies that $\ev[\xi] = \ev[\tilde{X}_t] = \ev[\xi] - b_t + \gamma t$ for $t \geq 0$, so that $b_t = \gamma t$. However, the unique stationary distribution of a reflected Brownian motion on the positive half-line with constant drift $-\gamma$ is $\textup{Exp}(2\gamma)$. 

We shall not discuss convergence to stationarity here, though we note that one can likely borrow ideas from the proof of \cite[Theorem 1.16]{baker_loc_times_2025} to establish this.
\end{remark}

We will now detail the sequence of arguments involved in the proof of Theorem \ref{thm:mfl_exist}. We start by outlining the connection between the mean-field limit \eqref{eq:mfl} and the supercooled Stefan problem.

\subsection{Connection with the Supercooled Stefan Problem}

Let us begin by providing a heuristic derivation of the Fokker--Planck equation satisfied by the strong formulation of McKean--Vlasov SDE \eqref{eq:mfl}. Suppose that $b$ solves \eqref{eq:mfl}, let $(X - b, L)$ denote the solution of the Skorokhod problem for $\xi + W - b$, and set $\mu_t = \L(X_t)$. Suppose that $b$ is differentiable and that $\mu_t$ has a regular density, which we shall denote by the same symbol. Then, $\mu_t$ satisfies the Fokker--Planck equation
\begin{equation*}
    % \partial_t \mu_t(x) = \dot{b}_t \partial_x \mu_t(x) + \frac{1}{2} \partial_x^2 \mu_t(x)
    \partial_t \mu_t(x) = \frac{1}{2} \partial_x^2 \mu_t(x)
\end{equation*}
for $x \in (b_t, \infty)$. Next, let us heuristically derive an appropriate boundary condition at $x = b_t$. It follows from the occupation time formula (see e.g.\@ \cite[Chapter VI, Corollary 1.6]{revuz_cmbm_1999}) that
\begin{equation*}
    \gamma t = \ev[L_t] = \frac{1}{2} \int_0^t \mu_s(b_s) \, \d s,
\end{equation*}
implying the Dirichlet boundary condition $\mu_t(b_t) = 2\gamma$. Furthermore, from conservation of mass, we can deduce that
\begin{align*}
    0 &= \frac{\d}{\d t} \int_{b_t}^{\infty} \mu_t(x) \, \d x \\
    &= - \dot{b}_t \mu_t(b_t) + \int_{b_t}^{\infty} \frac{1}{2} \partial_x^2 \mu_t(x) \, \d x \\
    &= -\dot{b}_t \mu_t(b_t) - \frac{1}{2} \partial_x \mu_t(b_t) \\
    &= -2\gamma \dot{b}_t - \frac{1}{2} \partial_x \mu_t(b_t),
\end{align*}
where we inserted the boundary condition in the last step. Rearranging yields that $\dot{b}_t = -\frac{1}{4\gamma} \partial_x \mu_t(b_t)$. Thus, $\mu_t$ satisfies the Fokker--Planck equation
\begin{equation} \label{eq:fpe}
\begin{cases}
    \partial_t \mu_t(x) = \frac{1}{2} \partial_x^2 \mu_t(x) &\text{for } x \in (b_t, \infty),\, t \geq 0, \\
    \mu_t(b_t) = \frac{1}{\alpha} &\text{for } t \geq 0, \\
    \dot{b}_t = -\frac{\alpha}{2} \partial_x \mu_t(b_t) &\text{for } t \geq 0,
\end{cases}
\end{equation}
where $\alpha = \frac{1}{2 \gamma}$. This is a moving boundary problem of Stefan-type and appeared in work by Cabezas, Dembo, Sarantsev \& Sidoravicius \cite{cabezas_rank_2019} on a scaling limit for the infinite Atlas model. The infinite Atlas model consists of a semi-infinite collection $(X^i)_{i \in \bb{N}}$ of particles, started from the points of a Poisson process with rate $\lambda > 0$. Similar to \eqref{eq:ps}, the lowest ranked amongst all particles receives a drift of intensity $\gamma$. Applying the diffusive scaling $X^{N, i}_t = \frac{1}{N} X^i_{N^2 t}$, one obtains an infinite version of Equation \eqref{eq:ps}. Cabezas, Dembo, Sarantsev \& Sidoravicius \cite{cabezas_rank_2019} describe the limit of this system as $N \to \infty$ through a version of PDE \eqref{eq:fpe} whose initial condition has constant density $\lambda$ and, consequently, infinite mass. The very recent work by Atar \& Budhiraja \cite{atar_atlas_2025} considers arbitrary locally finite measures as initial densities. We provide a detailed comparison of our results with those of \cite{cabezas_rank_2019} and \cite{atar_atlas_2025} in Subsection \ref{sec:literature}.

At first glance, PDE \eqref{eq:fpe} looks like a classical one-phase Stefan problem with inhomogeneous Dirichlet boundary condition $\mu_t(b_t) = \frac{1}{\alpha}$.
% , but with the inhomogeneous Dirichlet boundary condition $\mu_t(b_t) = \frac{1}{\alpha}$ instead of the more prominent homogeneous condition $\mu_t(b_t) = 0$. 
As such, it may seem more benign than the supercooled Stefan problem with homogeneous Dirichlet boundary condition, where the dynamics of the boundary $b_t$ carry the opposite sign, i.e.\@ $\dot{b}_t = \alpha \partial_x \mu_t(b_t)$. Indeed, for the supercooled Stefan problem, a homogeneous Dirichlet boundary condition implies that the gradient $\partial_x \mu_t(b_t)$ of the solution at the boundary is nonnegative. Hence, the boundary advances towards the particles, which can lead to a critical build-up of mass near the boundary, resulting in a blow-up. 

Now, if the gradient $\partial_x \mu_t(b_t)$ were also nonnegative in PDE \eqref{eq:fpe}, then, since the boundary dynamics have the opposite sign, the boundary would recede from the particles, so a blow-up would not be expected. However, as $\mu_t(b_t) = \frac{1}{\alpha} > 0$ in PDE \eqref{eq:fpe}, the gradient $\partial_x \mu_t(b_t)$ may be both positive and negative, corresponding to the fact that the boundary $b_t$ can both increase and decrease. Hence, the positive feedback present in the supercooled Stefan problem may also arise in PDE \eqref{eq:fpe}. In fact, it turns out that PDE \eqref{eq:fpe} is the supercooled Stefan problem in disguise. This point can be made transparent by applying a simple transformation to $\mu_t$.

Let us define the signed measure $\nu_t$ on $\R$ by $\nu_t(x) = \frac{1}{\alpha} - \mu_t(x)$ if $x \geq b_t$ and $\nu_t(x) = 0$ if $x < b_t$. Here we again identify $\nu_t$ with its density. To be clear, by a signed measure $m$ on some measurable space, we mean that $m = m_+ - m_-$ for measures $m_+$, $m_-$ on the same measurable space such that $m_+$ or $m_-$ is finite. Then, $\nu_t$ satisfies the PDE 
\begin{equation} \label{eq:sfp}
    \begin{cases}
    \partial_t \nu_t(x) = \frac{1}{2} \partial_x^2 \nu_t(x) &\text{for } x \in (b_t, \infty),\, t \geq 0, \\
    \nu_t(b_t) = 0 &\text{for } t \geq 0, \\
    \dot{b}_t = \frac{\alpha}{2} \partial_x \nu_t(b_t) &\text{for } t \geq 0.
\end{cases}
\end{equation}
This is precisely the supercooled Stefan problem. Since $\mu_t(x) \geq 0$, we have the upper bound $\nu_t(x) \leq \frac{1}{\alpha}$. Note, however, that in general, $\nu_t(x)$ may be negative. To get a better feeling for PDE \eqref{eq:sfp}, suppose that $\mu_0$ is nonincreasing and bounded from above by $\frac{1}{\alpha}$. Then, $\nu_0$ is nondecreasing, $\nu_0(0) = \frac{1}{\alpha} - \mu_0(0) < \frac{1}{\alpha}$, $0 \leq \nu_0(x) \leq \frac{1}{\alpha}$ for $x \in [0, \infty)$, and $\lim_{x \to \infty} \nu_0(x) = \frac{1}{\alpha} - \lim_{x \to \infty} \mu_0(x) = \frac{1}{\alpha}$. The inverse relationship between the (asymptotically attained) bound on the initial condition, namely $\frac{1}{\alpha}$, and the feedback coefficient $\alpha$ driving the boundary puts the equation exactly in the critical regime. In this setting, Baker, Hambly \& Jettkant \cite{baker_loc_times_2025} show that PDE \eqref{eq:sfp} has a unique global solution (in a suitably defined weak sense). In particular, the resulting boundary $b$ is a continuous nondecreasing function. If the feedback were any stronger (while $\lim_{x \to \infty} \nu_0(x)$ was kept fixed), then PDE \eqref{eq:sfp} would break down at a finite time horizon \cite[Theorem 1.12]{baker_loc_times_2025}. 

The construction of the global solution in the critical regime in \cite{baker_loc_times_2025} is based on a probabilistic representation of PDE \eqref{eq:sfp} in terms of a reflected McKean--Vlasov SDE. Our goal is to extend this probabilistic representation to the general case, where $\mu_0$ is not necessarily nonincreasing. Then, the only information we have on $\nu_0$ is that $\nu_0(0) < \frac{1}{\alpha}$, $\nu_0(x) \leq \frac{1}{\alpha}$ for $x \in (0, \infty)$, and $\lim_{x \to \infty} \nu_0(x) = \frac{1}{\alpha}$. Thus, we are still in the critical regime, but $\nu_0$ need not be nonnegative or nondecreasing. Neither (global) existence nor uniqueness is now guaranteed (or even considered) by \cite{baker_loc_times_2025}.

\subsection{Reformulation via McKean--Vlasov SDEs with Interaction Through Hitting and Local Times}

To motivate the representation of PDE \eqref{eq:sfp} through a reflected McKean--Vlasov SDE, let us first discuss its more canonical probabilistic representation in terms of a McKean--Vlasov SDE with interaction through hitting times \cite{hambly_mckean_2019, nadtochiy_mean_field_2020, delarue_stefan_2022}. This McKean--Vlasov SDE is given by
\begin{equation} \label{eq:probab_repr_hitting}
    Y^x_t = x + W_t - \alpha \ell_t, \qquad \ell_t = \int_{[0, \infty)} \pr(\tau_y \leq t) \, \d v(y)
\end{equation}
for $x \in [0, \infty)$, where $\tau_x = \inf\{t > 0 \define Y^x_t \leq 0\}$, and the initial condition $v$ is a signed measure. To obtain a representation of PDE \eqref{eq:sfp} in terms of this SDE, we must choose $v = \nu_0$. However, we intend to study McKean--Vlasov SDE \eqref{eq:probab_repr_hitting} in more generality, so we do not necessarily assume that $v$ is of the form $\nu_0 = \frac{1}{\alpha} - \mu_0$ for a probability measure $\mu_0$. Instead, we shall below (implicitly) formulate an assumption on the signed measure $v$, generalising Assumption \ref{ass:bounded_variation}, which guarantees that the integral on the right-hand side of \eqref{eq:probab_repr_hitting} is finite.

Formally, solutions to Equation \eqref{eq:probab_repr_hitting} are defined in the following sense.

\begin{definition} \label{def:probab_repr_hitting}
We say that a continuous function $\ell \define [0, \infty) \to \R$ is a solution to McKean--Vlasov SDE \eqref{eq:probab_repr_hitting} if
\begin{equation} \label{eq:hitting_sol_prop}
    \ell_t = \int_{[0, \infty)} \pr(\tau_x \leq t) \, \d v(x)
\end{equation}
for $t \geq 0$, where $\tau_x = \inf\{t > 0 \define x + W_t - \alpha \ell_t \leq 0\}$ for $x \in [0, \infty)$.

We refer to $v$ as the initial condition of McKean--Vlasov SDE \eqref{eq:probab_repr_hitting}.
\end{definition}

% Note that the integral on the right-hand side of \eqref{eq:hitting_sol_prop} is well-defined. Indeed, by the reflection principle, we have
% \begin{equation*}
%     \pr(\tau_x \leq t) \leq \pr\biggl(\sup_{0 \leq s \leq t} W_s \geq \frac{x - \alpha \ell_t}{\sqrt{t}}\biggr) = 2\pr\biggl(W_1 \geq \frac{x - \alpha \ell_t}{\sqrt{t}}\biggr).
% \end{equation*}
% The bound on the right-hand side above is integrable over $x \in [0, \infty)$, so that
% \begin{equation*}
%     \int_{[0, \infty)} \pr(\tau_x \leq t) \, \d \lvert \nu_0\rvert(x) = \frac{1}{\alpha} \int_{[0, \infty)} \pr(\tau_x \leq t) \, \d x + \int_{[0, \infty)} \pr(\tau_x \leq t) \, \d \mu_0(x) < \infty,
% \end{equation*}
% where $\lvert \nu_0\rvert$ denotes the variation of $\nu_0$. Thus, the integral in \eqref{eq:hitting_sol_prop} does exist and is finite. 
Note that in the setting where $v$ is a probability measure, it is now well-established (cf.\@ \cite{delarue_stefan_2022}) that McKean--Vlasov SDE \eqref{eq:probab_repr_hitting} is indeed a probabilistic representation of the supercooled Stefan problem, PDE \eqref{eq:sfp}, in the following sense: if $\ell$ solves McKean--Vlasov SDE \eqref{eq:probab_repr_hitting} and we define the distributions $\nu_t$, $t \geq 0$, on $\R$ by
\begin{equation*}
    \nu_t(A) = \int_{[0, \infty)} \pr\bigl((Y^x_t + \alpha \ell_t) \in A,\, \tau_x > t\bigr) \, \d v(x)
\end{equation*}
for Borel-measurable $A \subset [0, \infty)$ and $Y^x$ as in Equation \eqref{eq:probab_repr_hitting}, then $\nu_t$ is a solution to PDE \eqref{eq:sfp} and $b_t = \alpha \ell_t$. Of course, some care must be taken in formalising the precise sense in which $\nu_t$ solves PDE \eqref{eq:sfp}, since $\partial_x \nu_t(b_t)$ may not exist for all times. For details, we refer to \cite[Theorem 1.1]{delarue_stefan_2022}. In our analysis, we shall pass directly from McKean--Vlasov SDE \eqref{eq:probab_repr_hitting} to the mean-field limit \eqref{eq:mfl} of the Atlas model, circumventing the need for any regularity analysis of PDE \eqref{eq:sfp}. Let us note that if $\partial_x \nu_t(b_t)$ does exist for all times and is integrable, then 
\begin{equation*}
    \ell_t = \frac{1}{2} \int_0^t \partial_x \nu_s(b_s) \, \d s,
\end{equation*}
so $\ell_t$ is the cumulative flux across the absorbing boundary $b_t = \alpha \ell_t$.

McKean--Vlasov SDE \eqref{eq:probab_repr_hitting} is not yet our target representation for \eqref{eq:mfl}. The target rather corresponds to the spatial derivative of the supercooled Stefan problem. To make sense of this derivative, we must impose additional regularity on $v$. This is where Assumption \ref{ass:bounded_variation} comes into play. If $v = \nu_0$ and Assumption \ref{ass:bounded_variation} is satisfied, then the Hahn decomposition theorem implies that we can find a finite signed measure $m$ on $[0, \infty)$ such that $m([0, x]) = \nu_0(x) = \frac{1}{\alpha} - \mu_0(x)$ for $x \in [0, \infty)$. In our more general setting, where $v$ is an arbitrary signed measure, we shall instead assume that there exists another signed measure $m$ on $[0, \infty)$ such that $m([0, x]) = v(x)$ for $x \in [0, \infty)$. Note that here we identify $v$ with its density. Then, we consider the McKean--Vlasov SDE
\begin{equation} \label{eq:probab_repr_reflected}
    X^x_t = x + W_t - \alpha \ell_t + L^x_t, \qquad \ell_t = \int_{[0, \infty)} \ev[L^y_t] \, \d m(y)
\end{equation}
for $x  \in [0, \infty)$, with reflection at the origin. While McKean--Vlasov SDE \eqref{eq:probab_repr_hitting} features interaction through the law of hitting times, in the above equation, interaction occurs through the law of local times. McKean--Vlasov SDE \eqref{eq:probab_repr_reflected} is understood in the following sense.

\begin{definition} \label{def:probab_repr_reflected}
We say that a continuous function $\ell \define [0, \infty) \to \R$ is a solution to McKean--Vlasov SDE \eqref{eq:probab_repr_reflected} if 
\begin{equation} \label{eq:probab_repr_reflected_sol_prop}
    \ell_t = \int_{[0, \infty)} \ev[L^x_t] \, \d m(x)
\end{equation}
for $t \geq 0$, where $(X^x, L^x)$ solves the Skorokhod problem for $x + W - \alpha \ell$ for $x \in [0, \infty)$.

We refer to $m$ as the initial condition of McKean--Vlasov SDE \eqref{eq:probab_repr_reflected}.
\end{definition}

Under the following assumption on $m$ (and, by extension, $v$), the integrals on the right-hand side of Equations \eqref{eq:hitting_sol_prop} and \eqref{eq:probab_repr_reflected_sol_prop} are finite. 

\begin{assumption} \label{ass:integrability}
We assume that $m$ is a locally finite signed measure on $[0, \infty)$ such that $m(\{0\}) < \frac{1}{\alpha}$, $m([0, x]) \leq \frac{1}{\alpha}$ for $x \geq 0$, and for all $c > 0$, it holds that
\begin{equation} \label{eq:ic_growth}
    \int_{[0, \infty)} e^{-c x^2} \, \d \lvert m\rvert(x) < \infty.
\end{equation}
Here $\lvert m\rvert$ is the variation of $m$.
\end{assumption}

Note that when $m$ is derived from the density of the distribution of $\xi$ using Assumption \ref{ass:bounded_variation}, then it satisfies Assumption \ref{ass:integrability}. However, if we only suppose that Assumption \ref{ass:integrability} holds, then the nonnegative function $x \mapsto (\frac{1}{\alpha} - m([0, x])) = (\frac{1}{\alpha} - v(x))$ need not even integrate to a finite value, so it does not necessarily determine a probability distribution.

As mentioned earlier, McKean--Vlasov SDE \eqref{eq:probab_repr_reflected} should be understood as a probabilistic representation of the derivative of the supercooled Stefan problem. What we mean by that is the following: if the solution $\nu_t$ of PDE \eqref{eq:sfp} were sufficiently regular, we should expect that
\begin{equation}
    \int_A \partial_x \nu_t(x) \, \d x = \int_{[0, \infty)} \pr\bigl((X^x_t + \alpha \ell_t) \in A\bigr) \, \d m(x)
\end{equation}
for Borel measurable $A \subset [0, \infty)$. In the case where $m$ is a finite measure, this is explained in more detail in \cite[Subsection 1.5.2]{baker_loc_times_2025}. Here, we shall exclusively focus on the probabilistic perspective and show that McKean--Vlasov SDEs \eqref{eq:probab_repr_hitting} and \eqref{eq:probab_repr_reflected} have the same set of solutions. This follows from the easily established equality
\begin{equation} \label{eq:local_to_hitting_intro}
    \ev\biggl[\sup_{0 \leq s \leq t} (x + W_s - \alpha \ell_s)_-\biggr] = \int_0^{\infty} \pr\biggl(\inf_{0 \leq s \leq t} (y + W_s - \alpha \ell_s) \leq 0\biggr) \, \d y,
\end{equation}
which connects local times with hitting times (cf.\@ Lemma \ref{lem:integrals_finite}). 

\begin{proposition} \label{prop:equivalence}
Let Assumption \ref{ass:integrability} be satisfied. Then, a function $\ell \in C([0, \infty))$ solves McKean--Vlasov SDE \eqref{eq:probab_repr_hitting} if and only if it solves McKean--Vlasov SDE \eqref{eq:probab_repr_reflected}. 
\end{proposition}

In the statement of the above proposition, it is understood that the initial condition $v$ of \eqref{eq:probab_repr_hitting} and the initial condition $m$ of \eqref{eq:probab_repr_reflected} are related by $v(x) = m([0, x])$ for $x \geq 0$.

Under Assumption \ref{ass:integrability}, we can show that McKean--Vlasov SDE \eqref{eq:probab_repr_reflected} has a solution. % In fact, we will prove a more general existence result that also covers scaling limits of the infinite Atlas model. 
In view of Proposition \ref{prop:equivalence}, this also implies existence for McKean--Vlasov SDE \eqref{eq:probab_repr_hitting}.

\begin{theorem} \label{thm:repr_reflected_exist}
Let Assumption \ref{ass:integrability} be satisfied. Then McKean--Vlasov SDE \eqref{eq:probab_repr_reflected} has a solution.
\end{theorem}

The growth condition \eqref{eq:ic_growth} is necessary to ensure that the integral on the right-hand side of \eqref{eq:probab_repr_reflected_sol_prop} is finite (see Lemma \ref{lem:integrals_finite}). In that sense, the assumptions of Theorem \ref{thm:repr_reflected_exist} are essentially minimal for existence in the (sub)critical case. Recall that by critical (subcritical) we mean that the coefficient $\alpha$ appearing in \eqref{eq:probab_repr_reflected} and the maximal mass $\sup_{x \geq 0} m([0, x])$ satisfy $\alpha \sup_{x \geq 0} m([0, x]) \leq 1$ ($\alpha \sup_{x \geq 0} m([0, x]) < 1$). Theorem \ref{thm:repr_reflected_exist} shows that in the (sub)critical case, McKean--Vlasov SDE \eqref{eq:probab_repr_reflected} has a continuous global solution. In the supercritical regime,
\begin{equation*}
    \alpha \sup_{x \geq 0} m([0, x]) > 1,
\end{equation*}
solutions may only exist locally or exhibit jumps. A case where continuous but only local-in-time solutions exist is discussed in \cite[Theorem 1.12]{baker_loc_times_2025} under the assumption that $m$ is a probability measure. The supercooled Stefan problem with nonnegative and integrable initial condition is a special case in which global solutions exist but jumps may occur \cite{hambly_mckean_2019, nadtochiy_mean_field_2020, delarue_stefan_2022}.

The proof of Theorem \ref{thm:repr_reflected_exist} is based on an application of the Schauder fixed-point theorem. Since we are in the critical regime, the main challenge lies in finding a stable set $\cal{K} \subset C([0, \infty))$ for the fixed-point map whose image under the fixed-point map is contained in a compact subset of $\cal{K}$. Due to the criticality, one might worry that elements of $\cal{K}$ grow unboundedly under repeated application of the fixed-point map and, thus, escape from $\cal{K}$. To construct a suitable $\cal{K}$, we employ as an upper bound the global solution to McKean--Vlasov SDE \eqref{eq:probab_repr_reflected} obtained in Baker, Hambly \& Jettkant \cite{baker_loc_times_2025} for initial conditions given by finite measures. The applicability of this theory crucially rests on the fact that $m(\{0\}) < \frac{1}{\alpha}$ and $m([0, x]) \leq \frac{1}{\alpha}$ for $x \geq 0$.

Note that the Schauder fixed-point theorem only provides existence of a solution. For uniqueness, we shall exploit the correspondence of McKean--Vlasov SDE \eqref{eq:probab_repr_reflected} with McKean--Vlasov SDE \eqref{eq:mfl}.
% The proof of Theorem \ref{thm:repr_reflected_exist} proceeds via successive extension of local-in-time solutions constructed through a contraction argument. To ensure that no blow-up can occur in finite time, it must be shown that the signed measure $\int_{[0, \infty)} \ev[\delta_{X^x_t}] \, \d m(x)$, describing the distribution of mass at time $t \geq 0$, remains subcritical. That is, the mass concentrated at the origin stays below the critical threshold $\frac{1}{\alpha}$.
Indeed, as mentioned above, combining Theorem \ref{thm:repr_reflected_exist} with Proposition \ref{prop:equivalence} implies that McKean--Vlasov SDE \eqref{eq:probab_repr_hitting} has a solution $\ell$. Then, tracing back our earlier steps, we would expect that $b = \alpha \ell$ is a solution to McKean--Vlasov SDE \eqref{eq:mfl} in the strong sense. Unlike the equivalence between McKean--Vlasov SDEs \eqref{eq:probab_repr_hitting} and \eqref{eq:probab_repr_reflected}, which is straightforwardly demonstrated using the identity \eqref{eq:local_to_hitting_intro}, showing that McKean--Vlasov SDEs \eqref{eq:mfl} and \eqref{eq:probab_repr_hitting} are equivalent turns out to be more involved. For that, we draw on a superposition principle for reflected SDEs, which we establish in Proposition \ref{prop:superposition} below. 

\begin{theorem} \label{thm:equivalence}
Let Assumption \ref{ass:bounded_variation} be satisfied. Then, a function $b \in C([0, \infty))$ solves McKean--Vlasov SDE \eqref{eq:mfl} if and only if $\frac{1}{\alpha} b$ solves McKean--Vlasov SDE \eqref{eq:probab_repr_hitting}. In particular, McKean--Vlasov SDEs \eqref{eq:mfl}, \eqref{eq:probab_repr_hitting}, and \eqref{eq:probab_repr_reflected} have a unique solution.
\end{theorem}

The last statement of the above theorem follows upon combining the existence result for McKean--Vlasov SDE \eqref{eq:probab_repr_reflected} (Theorem \ref{thm:repr_reflected_exist}), the equivalence between McKean--Vlasov SDEs \eqref{eq:probab_repr_reflected} and \eqref{eq:probab_repr_hitting} (Proposition \ref{prop:equivalence}), and the equivalence between McKean--Vlasov SDEs \eqref{eq:probab_repr_hitting} and \eqref{eq:mfl} (established by the above theorem) with the uniqueness of McKean--Vlasov SDE \eqref{eq:mfl} in the strong sense (Proposition \ref{prop:mfl_strong_uniqueness}). 

Under the weaker Assumption \ref{ass:integrability}, it is not necessarily the case that the function $x \mapsto (\frac{1}{\alpha} - m([0, x]))$ integrates to one. Thus, we cannot connect McKean--Vlasov SDEs \eqref{eq:probab_repr_reflected} and \eqref{eq:probab_repr_hitting} with the mean-field limit \eqref{eq:mfl}, whose initial condition is a probability distribution, and transfer the uniqueness result for the latter to the former two equations.

\subsection{Related Literature} \label{sec:literature}

There is a rich literature on ranked-based particle systems both with a finite and a countably infinite number of particles, dating back to works of Harris \cite{harris_collisions_1965} and Sznitman \cite{sznitman_poc_1991} on rankings of Brownian motion. Much attention has been dedicated to the semi-infinite Atlas model, consisting of a countable collection of particles, started from the points of a Poisson process, that diffuse according to independent Brownian motions, with only the lowest ranked particle assigned a positive drift. Pal \& Pitman \cite{pal_ranked_bm_2008} study the long-range behaviour of the spacings between particles, showing convergence to a stationary distribution given by independent exponential distributions. They also hypothesise that a suitable scaling of the $i$th ranked particle converges to a fractional Brownian motion with Hurst parameter $\frac{1}{4}$. The validity of this conjecture follows from results by Dembo \& Tsai \cite{dembo_atlas_2017} who prove that rescaled fluctuations of the Atlas model around its equilibrium follow a stochastic heat equation with Neumann boundary condition.  

As we already alluded to above, Cabezas, Dembo, Sarantsev \& Sidoravicius \cite{cabezas_rank_2019} analyse the large $N$ limit of the semi-infinite analogue of \eqref{eq:ps}, where particles are started from a constant-rate Poisson process. Due to the simple initial profile, they can explicitly determine the system's limiting density, which solves a Stefan-type moving boundary problem. Depending on whether the rate of the Poisson process lies within $(0, 2\gamma) = (0, \frac{1}{\alpha})$ or is above $2\gamma = \frac{1}{\alpha}$, the deterministic limit of the Atlas particle, which plays the role of the moving boundary, follows a square-root trajectory with a positive or negative coefficient. 

The more recent work by Atar \& Budhiraja \cite{atar_atlas_2025}, generalises \cite{cabezas_rank_2019} to varying initial configurations provided by locally finite measures. They characterise the limiting profile through a moving boundary problem involving measures that generalises the one-phase Stefan problem from \cite{cabezas_rank_2019}. In particular, the limit of the Atlas particle is no longer shown to follow a deterministic trajectory but is instead represented by a measure. This is analogous and served as an inspiration for our generalised formulation of McKean--Vlasov SDE \eqref{eq:mfl} from Definition \ref{def:generalised_solution}. Only if the initial configuration $\mu_0$ is sufficiently dense in the sense that $\d \mu_0(x) \geq \lambda \, \d x$ for some $\lambda > 0$, a deterministic trajectory can be recovered. This result is obtained by exploiting the fact that the gaps of the particle system with the dense initial profile $\mu_0$ are stochastically dominated by those from the system with the constant initial density $\lambda$. Then, one can proceed similarly to \cite{cabezas_rank_2019} to conclude. Note that our Assumption \ref{ass:bounded_variation} does not allow for a comparison argument, so the proof techniques we employ are quite different from those in \cite{atar_atlas_2025}. Let us emphasise though that our assumptions are not weaker than the ones from \cite{atar_atlas_2025}, so the results are (partially) complementary.  

Finite ranked-based particle systems were studied in much more generality than their infinite counterparts, with rank-dependence both in the drift and diffusion coefficients of all particles being considered. The first comprehensive existence result for weak solutions of such equations was provided by Bass \& Pardoux \cite{bass_uniqueness_piecewise_1987}, though their proper motivation comes from piecewise filtering. The applicability to particle systems with interaction through ranks is rather coincidental. Subsequently, Ichiba, Karatzas \& Shkolnikov \cite{ichiba_ranked_2012} showed that the weak solutions are in fact strong up until the first time three particles collide. Criteria for the absence and presence of such collision events were derived in \cite{ichiba_ranked_2010, ichiba_ranked_2012, sarantsev_triple_2015}. Banner, Fernholz \& Karatzas \cite{banner_atlas_2005} study the long-term behaviour of generalised Atlas models with finitely many particles. Corresponding results for the simple semi-infinite Atlas model can be found in \cite{pal_ranked_bm_2008, sarantsev_atlas_2017}.

Several articles consider propagation of chaos for particle systems where the rank-dependence of the coefficients is expressed through a dependence on the empirical cumulative distribution function (CDF) of the particles. Note that this structural requirement precludes the Atlas model due to the exploding drift of lowest ranked particle. In many cases, the limit of the empirical CDF can be shown to satisfy a generalised one-dimensional porous medium equation. The solution of this equation describes the CDF of the representative particle in the mean-field limit. One of the earliest works in this direction is that of Jourdain \cite{jourdain_pme_2000}, who studies particle approximations of the classical porous medium equation. Jourdain establishes propagation of chaos for the approximating ranked-based particle system and proves that the limit of the empirical CDF is indeed a solution of the classical porous medium equation. These results were extended to a broader range of coefficients by Jourdain \& Reygner \cite{jourdain_poc_rank_2013}. In a preceding work, Shkolnikov \cite{shkolnikov_ranked_based_2012} proves propagation of chaos for initial configurations for which the gaps between consecutive particle ranks are stationary. Since then various extensions have appeared: large deviations \cite{dembo_ranked_ld_2016}, a central limit theorem \cite{kolli_clt_rank_2018}, and the addition of common noise \cite{kolli_ranked_common_2019, shkolnikov_rank_2024}.

Lastly, let us comment on the connection between the Atlas model and Stefan-type moving boundary problems (see e.g.\@ \cite{cannon_sp_1970, sherman_sp_1970, fasano_fbp_1977, chayes_hdl_1996} and references therein). As discussed earlier, this connection was already explored in \cite{cabezas_rank_2019} for initial densities with constant value $\lambda > 0$. If $\lambda > 2\gamma = \frac{1}{\alpha}$, the boundary in PDE \eqref{eq:fpe} is nonincreasing, corresponding to the classical Stefan problem describing the melting of a frozen liquid. For $\lambda \in (0, 2\gamma)$, the boundary advances into the domain, as is the case for the supercooled Stefan problem that models freezing of a supercooled liquid. For nonconstant initial conditions, treated in \cite{atar_atlas_2025} and the present work, the boundary need not be monotonic, yielding a mixture of both models. While \cite{cabezas_rank_2019} and \cite{atar_atlas_2025} work directly with PDE \eqref{eq:fpe} (or in the latter case with an integrated and relaxed version introduced by Atar \cite{atar_weak_fbp_2025}), we proceed via the transformed PDE \eqref{eq:sfp}. The latter leads us to a precise relationship between the Atlas model and the supercooled Stefan problem from \cite{hambly_mckean_2019, nadtochiy_mean_field_2020}, which admits the convenient probabilistic representation \eqref{eq:probab_repr_hitting}. Solutions to this representation can be constructed by extending ideas from Baker, Hambly \& Jettkant \cite{baker_loc_times_2025}.

\subsection{Main Contributions and Structure of the Paper}

Let us conclude this section with a discussion of our contributions and an outline of the paper. The first contribution we wish to highlight is our new probabilistic interpretation of the mean-field limit of the Atlas model as a novel type of reflected SDE with a constraint on the mean-growth of the regulator process. We introduce a relaxed formulation of this SDE and rigorously connect it with various existing models.

Our analysis begins in Section \ref{sec:finite} with a study of the finite particle system \eqref{eq:ps}. We discuss its reformulation through the reflected SDE \eqref{eq:ps_reflected}, for which we establish a comparison result. This turns out to be useful when we subsequently prove the tightness of the particle system. The main challenge here lies in obtaining sufficient control on the increments of the processes
\begin{equation*}
    L^{N, i} = \int_0^{\cdot} \gamma N \bf{1}_{\{X^{N, i}_t = B^N_t\}} \, \d t,
\end{equation*}
for $N \geq i$, that guarantee tightness on $C([0, \infty))$. We achieve this by proving that as $N \to \infty$, the distribution of particles becomes sufficiently dense, so that none of the particles can be the Atlas particle for sustained periods of time.

Section \ref{sec:mfl} serves a dual purpose. First, we show that any limit point of the particle system \eqref{eq:ps} is a solution to McKean--Vlasov SDE \eqref{eq:mfl} in the generalised sense, thereby completing the proof of Theorem \ref{thm:convergence}. Then, we prove the uniqueness result for generalised solutions, Theorem \ref{thm:unique_generalised}, drawing on approximations by solutions of \eqref{eq:mfl} in the strong sense. Unlike \cite{atar_atlas_2025}, our arguments are probabilistic in nature, avoiding an analysis of PDEs \eqref{eq:fpe} and \eqref{eq:sfp}.

In Section \ref{sec:hitting_reflected}, we analyse the McKean--Vlasov SDEs \eqref{eq:probab_repr_hitting} and \eqref{eq:probab_repr_reflected} with interaction through hitting and local times. First, we show the equivalence between McKean--Vlasov SDEs \eqref{eq:probab_repr_hitting} and \eqref{eq:probab_repr_reflected} (Proposition \ref{prop:equivalence}). Then, we construct a solution to the latter, giving Theorem \ref{thm:repr_reflected_exist}, and yielding, simultaneously, a solution to the former.

The objective of the final Section \ref{sec:strong_existence} is to establish the existence of solutions to McKean--Vlasov SDE \eqref{eq:mfl} in the strong sense (Theorem \ref{thm:mfl_exist}). This is done by showing that McKean--Vlasov SDE \eqref{eq:probab_repr_hitting} and the strong formulation of the mean-field limit are equivalent (Theorem \ref{thm:equivalence}). Since existence for the former equation was derived in Section \ref{sec:hitting_reflected}, we thus obtain a solution to the mean-field limit. The equivalence result is based on our superposition principle for reflected SDEs, Proposition \ref{prop:superposition}.

\section{The Finite Atlas Model} \label{sec:finite}

In this section, we analyse the particle system \eqref{eq:ps} through the lens of the reflected SDE \eqref{eq:ps_reflected} whose boundary is determined by a constraint on the growth of the empirical average of the regulator processes. We prove several comparison results for SDE \eqref{eq:ps_reflected}, which, in particular, imply uniqueness. These are subsequently used to establish tightness of the particle system. 

Let us begin by defining solutions to SDE \eqref{eq:ps_reflected} for a given $N \geq 1$. For that we fix a probability space $(\Omega, \F, \pr)$ equipped with a filtration $\bb{F}^N = (\F^N_t)_{t \geq 0}$, independent $\F^N_0$-measurable random variables $\xi_1$,~\ldots, $\xi_N$, and independent $\bb{F}^N$-Brownian motions $W^1$,~\ldots, $W^N$. % We denote by $\bb{F}^N = (\F^N_t)_{t \geq 0}$ the filtration given by $\F^N_t = \sigma(\xi_i, W^i_s \define s \in [0, t],\, i \in [N])$.
For notational convenience, we set $[N] = \{1, \dots, N\}$.

\begin{definition} \label{def:ps_reflected}
We say that a continuous $\bb{F}^N$-adapted process $B = (B_t)_{t \geq 0}$ is a solution of SDE \eqref{eq:ps_reflected} if the solutions $(\tilde{X}^i, L^i)$, $i \in [N]$, of the Skorokhod problem for $\xi_i + W^i - B$ satisfy $\frac{1}{N} \sum_{i = 1}^N L^i_t = \gamma t$ for $t \geq 0$.
\end{definition}

Let $X^1$,~\ldots, $X^N$ be a solution to the Atlas model \eqref{eq:ps}, guaranteed to exist by the theory of Bass \& Pardoux \cite{bass_uniqueness_piecewise_1987}. Note that this solution is a priori weak, so we tacitly assume that $(\Omega, \F, \bb{F}^N, \pr)$ is sufficiently rich to support $X^1$,~\ldots, $X^N$. We already saw in the introduction that setting $B^N = X^{(1)}$ yields a solution to SDE \eqref{eq:ps_reflected}. Moreover, the corresponding solutions of the Skorokhod problems for $\xi_i + W^i - B^N$ is given by $(X^i - B^N, L^i)$, where $L^i = (L^i_t)_{t \geq 0}$ is defined as $L^i_t = \int_0^t \gamma N \bf{1}_{\{X^i_s = B^N_s\}} \, \d s$. According to the following comparison result, this is also the only solution to SDE \eqref{eq:ps_reflected}.

\begin{proposition} \label{prop:comparison}
For $k = 1$, $2$, let $N_k \geq 1$ and let $B^k$ be a solution to SDE \eqref{eq:ps_reflected} with $\F^{N_k}_0$-measurable initial conditions $\xi^k_1$,~\ldots, $\xi^k_{N_k}$ and denote the solution to the Skorokhod problem for $\xi^k_i + W^i - B^k$ by $(X^{k, i} - B^k, L^{k, i})$, $i = 1$,~\ldots, $N_k$. 
\begin{enumerate}[noitemsep, label = (\roman*)]
    \item \label{it:comparison} If $N_1 = N_2$ and $\xi^1_i \leq \xi^2_i$ for $i \in [N_1]$, then we have that $X^{1, i}_t \leq X^{2, i}_t$ and $B^1_t \leq B^2_t$ for $t \geq 0$ and $i \in [N_1]$. In particular, SDE \eqref{eq:ps_reflected} exhibits pathwise uniqueness.
    \item \label{it:comparison_local} If $N_1 = N_2$, $\xi^1_i \leq \xi^2_i$ for some $i \in [N_1]$, and $\xi^1_j = \xi^2_j$ for $j \in [N_1] \setminus \{i\}$, then $L^{1, i}_t \geq L^{2, i}_t$ for $t \geq 0$.
    \item \label{it:comparison_number} If $N_1 \leq N_2$ and $\xi^1_i = \xi^2_i$ for $i \in [N_1]$, then we have that $L^{1, i}_t \geq L^{2, i}_t$ for $t \geq 0$ and  $i \in [N_1]$.
\end{enumerate}
\end{proposition}

\begin{proof}
We begin by establishing \ref{it:comparison}. The proof is similar to that of Proposition \ref{prop:mfl_strong_uniqueness}, which concerns the mean-field limit. For notational simplicity, set $N = N_1 = N_2$ and let $B^k$, $X^{k, i}$, and $L^{k, i}$, $i \in [N]$ and $k \in \{1, 2\}$ be as in the statement of the proposition. Then by the fundamental theorem of calculus, we have
\begin{align} \label{eq:pos_part}
    (X^{1, i}_t - X^{2, i}_t)_+^2 &= (\xi^1_i - \xi^2_i)_+^2 + 2\int_0^t (X^{1, i}_s - X^{2, i}_s)_+ \, \d (L^{1, i}_s - L^{2, i}_s) \notag \\
    &\leq 2\int_0^t (B^1_s - B^2_s)_+ \, \d (L^{1, i}_s - L^{2, i}_s),
\end{align}
where we used that $\xi^1_i \leq \xi^2_i$ and the properties of the Skorokhod problem in the second line. Summing the above over $i = 1$,~\ldots, $N$ implies that
\begin{equation}
    \sum_{i = 1}^N (X^{1, i}_t - X^{2, i}_t)_+^2 \leq 2\sum_{i = 1}^N \int_0^t (B^1_s - B^2_s)_+ \, \d (L^{1, i}_s - L^{2, i}_s) = 0
\end{equation}
since $\sum_{i = 1}^N L^{1, i}_t = \sum_{i = 1}^N L^{2, i}_t = \gamma N t$. It follows that $X^{1, i}_t - X^{2, i}_t \leq 0$ for $t \geq 0$ and $i \in [N]$. 

In order to show that the barriers are ordered as well, we shall exploit the correspondence of SDE \eqref{eq:ps_reflected} with the Atlas model \eqref{eq:ps}. To argue this way, we must first verify that the solution to SDE \eqref{eq:ps_reflected} provided by the Atlas model is the only one. To achieve this, we simply prove uniqueness of SDE \eqref{eq:ps_reflected}. Note that if the initial conditions $\xi^1_i$ and $\xi^2_i$ coincide for $i \in [N]$, we can apply the previous comparison argument in both directions, implying that $X^{1, i}_t = X^{1, i}_t$ and, therefore,
\begin{equation*}
    L^{1, i}_t = X^{1, i}_t - \xi^1_i - W^i_t = X^{2, i}_t - \xi^2_i - W^i_t = L^{2, i}_t
\end{equation*}
for $t \geq 0$ and for $i \in [N]$. Consequently, we have
\begin{equation*}
    \gamma \int_0^t B^1_s \, \d s = \frac{\gamma}{N} \sum_{i = 1}^N \int_0^t X^{1, i}_s \, \d L^{1, i}_s = \frac{\gamma}{N} \sum_{i = 1}^N \int_0^t X^{2, i}_s \, \d L^{2, i}_s = \gamma \int_0^t B^2_s \, \d s.
\end{equation*}
From the continuity of the barriers, we can therefore deduce that $B^1_t = B^2_t$ for $t \geq 0$. Thus, SDE \eqref{eq:ps_reflected} exhibits pathwise uniqueness and its unique solution is the one induced by the Atlas model.

Let us now finish the proof of the comparison principle. Since both solutions are the ones coming from the Atlas model, we have that
\begin{equation*}
    B^1_t = \min_{i \in [N]} X^{1, i}_t \leq \min_{i \in [N]} X^{2, i}_t = B^2_t
\end{equation*}
for $t \geq 0$. This concludes the proof of \ref{it:comparison}.

Let us show \ref{it:comparison_local} next, again setting $N = N_1 = N_2$. Similarly to \eqref{eq:pos_part}, we have for $k 
\geq 1$ and $t \geq 0$ that
\begin{align*}
    (X^{2, j}_t - X^{1, j}_t)^{2k} &= (\xi^2_j - \xi^1_j)^{2k} + 2k \int_0^t (X^{2, j}_s - X^{1, j}_s)^{2k - 1} \, \d (L^{2, j}_s - L^{1, j}_s) \\
    &\leq (\xi^2_j - \xi^1_j)^{2k} + 2k \int_0^t (B^2_s - B^1_s)^{2k - 1} \, \d (L^{2, j}_s - L^{1, j}_s)
\end{align*}
We sum this inequality over $j \in [N]$ and then raise both sides to the power $\frac{1}{2k}$ to obtain
\begin{equation*}
    \biggl(\sum_{j = 1}^N (X^{2, j}_t - X^{1, j}_t)^{2k}\biggr)^{\frac{1}{2k}} \leq \biggl(\sum_{j = 1}^N (\xi^2_j - \xi^1_j)^{2k} \biggr)^{\frac{1}{2k}}
\end{equation*}
Letting $k \to \infty$ in both expressions yields
\begin{equation*}
    X^{2, i}_t - X^{1, i}_t \leq \max_{j \in [N]} \lvert X^{2, j}_t - X^{1, j}_t \rvert \leq \max_{j \in [N]} \vert \xi^2_j - \xi^1_j\rvert = \xi^2_i - \xi^1_i.
\end{equation*}
From this we deduce that
\begin{equation*}
    \xi^2_i - \xi^1_i \geq X^{2, i}_t - X^{1, i}_t = \xi^2_i - \xi^1_i + L^{2, i}_t - L^{1, i}_t,
\end{equation*}
which finally implies that $L^{1, i}_t \geq L^{2, i}_t$. 

It remains to prove the last statement of the proposition. For $K \geq 1$, we introduce an auxiliary system with initial conditions $\tilde{\xi}^K_i$, $i \in [N_2]$, defined by $\tilde{\xi}^K_i = \xi^2_i$ for $i \in [N_1]$ and $\tilde{\xi}^K_i = \xi^2_i \lor K$ for $i \in [N_2] \setminus [N_1]$. We denote the corresponding solution to SDE \eqref{eq:ps_reflected} by $\tilde{B}^K$ and let $(\tilde{X}^{K, i} - \tilde{B}^K, \tilde{L}^{K, i})$ solve the Skorokhod problem for $\tilde{\xi}^K_i + W^i - \tilde{B}^K$. Since $\tilde{\xi}^K_i \geq \xi^2_i$ for $i \in [N_2]$, it holds by \ref{it:comparison} that $\tilde{X}^{K, i}_t \geq X^{2, i}_t$ for $t \geq 0$. But the initial conditions of the particles in $[N_1]$ coincide for the two systems of size $N_2$, so from $\tilde{X}^{K, i}_t \geq X^{2, i}_t$ we deduce that $\tilde{L}^{K, i}_t \geq L^{2, i}_t$ for $i \in [N_1]$. 

Let us now define $\varrho_K$ to be the first time $t \geq 0$ such that $\tilde{X}^{K, i}_t = \tilde{B}^K_t$ for some $i \in [N_2] \setminus [N_1]$. By pathwise uniqueness of SDE \eqref{eq:ps_reflected}, we have that $(\tilde{X}^{K, i}_t, \tilde{L}^{K, i}_t) = (X^{1, i}_t, L^{1, i}_t)$ for $t \in [0, \varrho_K)$. Thus, if we can show that $\varrho_K \to \infty$ a.s.\@ as $K \to \infty$, then it follows that
\begin{equation*}
    L^{1, i}_t = \lim_{K \to \infty} L^{1, i}_{t \land \varrho_K} = \lim_{K \to \infty} \tilde{L}^{K, i}_{t \land \varrho_K} \geq \lim_{K \to \infty} L^{2, i}_{t \land \varrho_K} = L^{2, i}_t
\end{equation*}
for $t \geq 0$ and $i \in [N_1]$. To see that $\varrho_K \to \infty$, we again appeal to the uniqueness of SDE \eqref{eq:ps_reflected}, whereby $\tilde{B}^K_t = B^1_t$ for $t \in [0, \varrho_K)$. Thus, on
\begin{equation*}
    \{\varrho_K > 0\} = \Bigl\{K > \min_{j \in [N_1]} \xi^1_j\Bigr\},
\end{equation*}
we can alternatively write $\varrho_K$ as the first time $t \geq 0$ that $\tilde{X}^{K, i} = B^1_t$ for some $i \in [N_2] \setminus [N_1]$. Now, the process $B^1$ is a.s.\@ locally bounded and, for $i \in [N_2] \setminus [N_1]$, we have $\tilde{X}^{K, i}_t \geq K + W^i_t$ for $t \geq 0$, so the first time that $\tilde{X}^{K, i}$ and $B^1$ meet indeed diverges as $K \to \infty$. This concludes the proof.
\end{proof}

Armed with Proposition \ref{prop:comparison}, we can establish tightness of the particle system.

\subsection{Tightness of the Atlas Model} \label{sec:tightness}

Let $B^N = (B^N_t)_{t \geq 0}$ be the unique solution to SDE \eqref{eq:ps_reflected} for $N \geq 1$ and let $(X^i - B^N, L^i)$ be the solution to the Skorokhod problem for $\xi_i + W^i - B^N$ for $i \in [N]$. Whenever we wish to emphasise the size of the particle system, we will add the number $N$ to the superscript of $X^i$, $L^i$, and other particle-related quantities. Recall that we defined the random measures $\beta^N$ on $[0, \infty) \times \R$ by
\begin{equation*}
    \d \beta^N(t, x) = \d \delta_{B^N_t} \d t.
\end{equation*}

Our goal is to prove that for any $i \in \bb{N}$, the sequence $(L^{N, i}, \beta^N)_{N \geq i}$ is tight on $C([0, \infty)) \times \cal{M}_1([0, \infty) \times \R)$. Here the space $\cal{M}_1([0, \infty) \times \R)$ of measures $m$ on $[0, \infty) \times \R$ such that $m([0, t] \times \R) = t$ for all $t \geq 0$ is endowed with the initial topology generated by the projections $\cal{M}_1([0, \infty) \times \R) \to \P([0, n] \times \R)$, $m \mapsto \frac{1}{n} m\vert_{[0, n] \times \R}$ for $n \in \bb{N}$. Note that by the properties of the initial topology, in order to show that a subset of $\cal{M}_1([0, \infty) \times \R)$ is compact, it suffices to show that the image of the subset under all the projections is compact. From this it follows that a sequence of random variables with values in $\cal{M}_1([0, \infty) \times \R)$ is tight if each projection of the sequence is tight.

\begin{proposition} \label{prop:tightness_ps}
For any $i \in \bb{N}$, the sequence $(L^{N, i}, \beta^N)_{N \geq i}$ is tight on $C([0, \infty)) \times \cal{M}_1([0, \infty) \times \R)$. Moreover, we have
\begin{equation} \label{eq:regulator_square_int}
    \sup_{N \geq i} \ev\bigl[\lvert L^{N, i}_t\rvert^2\bigr] < \infty
\end{equation}
for any $t \geq 0$.
\end{proposition}

\begin{proof}
We may establish the tightness of $(L^{N, i})_{N \geq i}$ and $(\beta^N)_{N \geq 1}$ separately. We begin with the former.

\textit{Tightness of} $(L^{N, i})_{N \geq i}$: It suffices to show that for any $T > 0$, the restriction of $L^{N, i}$, $N \geq i$, to the interval $[0, T]$ is tight on $C([0, T])$. So let us fix $T > 0$ and denote by $\cal{T}^N_{[0, T]}$ the set of $\bb{F}^N$-stopping times with values in $[0, T]$. Appealing to Aldous' tightness criterion (see e.g.\@ \cite[Theorem 16.10]{billingsley_convergence_1999}), we must show that for all $\eta > 0$,
\begin{equation} \label{eq:increment_regulator}
    \lim_{\theta \to 0} \limsup_{N \geq i} \sup_{\tau \in \cal{T}^N_{[0, T]}} \pr\Bigl(L^{N, i}_{\tau + \theta} - L^{N, i}_{\tau} > \eta\Bigr) = 0.
\end{equation}
The idea of the proof is as follows: first, we prove that for any $N \geq 1$ and any time $\tau \in \cal{T}^N_{[0, T]}$, there is a positive fraction of particles near $X^{N, i}_{\tau}$. Then, if $L^{N, i}$ were to grow rapidly between $\tau$ and $\tau + \theta$, all particles $j \in [N]$ sufficiently close to $X^{N, i}_{\tau}$ at time $\tau$ would be pushed upwards by an increase in their corresponding regulator process $L^{N, j}$. This in turn would contradict the fact that the empirical average of the regulators growth at the finite rate $\gamma$. Thus, the increments of $L^{N, i}$ cannot be too large as quantified by \eqref{eq:increment_regulator}.

Let us make the outlined strategy precise. Fix $\tau \in \cal{T}^N_{[0, T]}$ and $\theta \in [0, 1]$. Set
\begin{equation*}
    \sigma = \sup\bigl\{t \in [\tau, \tau + \theta] \define X^{N, i}_t = B^N_t\bigr\},
\end{equation*}
where $\sup \varnothing = \tau$. Then $L^{N, i}_{\tau + \theta} = L^{N, i}_{\sigma}$. Let $I$ be a nonempty random $\F^N_{\tau}$-measurable subset of $[N]$. On $\{\sigma > \tau\}$, we have that
\begin{align} \label{eq:growth_of_l}
    L^{N, i}_{\sigma} - L^{N, i}_{\tau} &= X^{N, i}_{\sigma} - X^{N, i}_{\tau} + W^i_{\sigma} - W^i_{\tau} \notag \\
    &\leq \frac{1}{\lvert I \rvert}\sum_{j \in I} \bigl(X^{N, j}_{\sigma} - X^{N, j}_{\tau}\bigr) + \frac{1}{\lvert I \rvert}\sum_{j \in I} \bigl(X^{N, j}_{\tau} - X^{N, i}_{\tau}\bigr) + W^i_{\sigma} - W^i_{\tau} \notag \\
    &\leq \frac{N}{\lvert I \rvert} \frac{1}{N} \sum_{j = 1}^N \bigl(L^{N, j}_{\sigma} - L^{N, j}_{\tau}\bigr) + \frac{1}{\lvert I \rvert}\sum_{j \in I} \bigl\lvert X^{N, j}_{\tau} - X^{N, i}_{\tau}\bigr\rvert \notag \\
    &\ \ \ + \frac{1}{\lvert I \rvert}\sum_{j \in I} (W^j_{\sigma} - W^j_{\tau}) + W^i_{\sigma} - W^i_{\tau} \notag \\
    &\leq \frac{\gamma \theta N}{\lvert I\rvert} + \frac{1}{\lvert I \rvert}\sum_{j \in I} \bigl\lvert X^{N, j}_{\tau} - X^{N, i}_{\tau}\bigr\rvert + \sup_{t \in [0, \theta]} M^I_t,
\end{align}
where $M^I_t = \frac{1}{\lvert I \rvert}\sum_{j \in I} (W^j_{\tau + t} - W^j_{\tau}) + W^i_{\tau + t} - W^i_{\tau}$ for $t \geq 0$. Here we used in the first inequality that on $\{\sigma > \tau\}$, it holds that $X^{N, i}_{\sigma} = B^N_{\sigma} \leq X^{N, j}_{\sigma}$. Note that the continuous martingale $M^I$ has quadratic variation
\begin{equation*}
    \langle M^I \rangle_t = \biggl(\frac{1 + 2\bf{1}_{\{i \in I\}}}{\lvert I\rvert} + 1\biggr) t.
\end{equation*}
Thus, conditional on the random set $I$, the process $M^I$ is a Brownian motion with variance $(1 + 2\bf{1}_{\{i \in I\}})\lvert I\rvert^{-1} + 1 \leq 4$, so that
\begin{equation*}
    \pr\biggl(\sup_{t \in [0, \theta]} M^I_t > \eta\biggr) \leq 2 \pr\biggl(W^1_1 > \frac{\eta}{2 \sqrt{\theta}}\biggr)
\end{equation*}
for $\eta > 0$. From this and \eqref{eq:growth_of_l}, we deduce for any nonempty $\F^N_{\tau}$-measurable subset $I$ of $[N]$ and any $\eta > 0$ that
\begin{align*}
    \pr\Bigl(L^{N, i}_{\sigma} - L^{N, i}_{\tau} > \eta\Bigr) &\leq \pr\biggl(\lvert I\rvert < \frac{\gamma \theta N}{\eta}\biggr) + \pr\biggl(\frac{1}{\lvert I \rvert}\sum_{j \in I} \bigl\lvert X^{N, j}_{\tau} - X^{N, i}_{\tau}\bigr\rvert > \eta\biggr) \\
    &\ \ \ + 2\pr\biggl(W^1_1 > \frac{\eta}{2 \sqrt{\theta}}\biggr).
\end{align*}
Now, let us choose $I = I^N_{\tau}$, where for $t \geq 0$, $I^N_t$ is the set of $j \in [N]$ such that $\tau^N_j = \inf\{s > 0 \define \lvert X^{N, j}_s - X^{N, i}_s\rvert \geq \eta\}$ occurs after $t$. By definition, for all $j \in I^N_{\tau}$, we have that $\lvert X^{N, j}_{\tau} - X^{N, i}_{\tau}\rvert \leq \eta$, so that
\begin{equation*}
    \pr\biggl(\frac{1}{\lvert I^N_{\tau} \rvert}\sum_{j \in I^N_{\tau}} \bigl\lvert X^{N, j}_{\tau} - X^{N, i}_{\tau}\bigr\rvert > \eta\biggr) = 0.
\end{equation*}
Since, moreover, we have $\pr(W^1_1 > \eta/(2\sqrt{\theta}) \to 0$ as $\theta \to 0$ and that $I^N_{\tau} \supset I^N_T$, in view of \eqref{eq:increment_regulator}, it remains to prove that
\begin{equation*}
    \lim_{\theta \to 0} \limsup_{N \geq i} \pr\biggl(\lvert I^N_T\rvert < \frac{\gamma \theta N}{\eta}\biggr) = \lim_{\delta \to 0} \limsup_{N \geq i} \pr\biggl(\frac{\lvert I^N_T\rvert}{N} < \delta\biggr) \to 0.
\end{equation*}

Fix $j \in [N]$ and set $Z^{N, j}_t = \lvert X^{N, j}_t - X^{N, i}_t\rvert$. By the It\^{o}--Tanaka formula (see \cite[Chapter 3, Theorem 7.1]{karatzas_bmsc_1998}), we have that
\begin{align} \label{eq:reflected_difference_comp}
    \d Z^{N, j}_t &= \sgn\bigl(X^{N, j}_t - X^{N, i}_t\bigr) \, \d (W^j_t - W^i_t) + \sgn\bigl(X^{N, j}_t - X^{N, i}_t\bigr) \, \d \bigl(L^{N, j}_t - L^{N, i}_t\bigr) + \d \Lambda^{N, j}_t \notag \\
    &= \sgn\bigl(X^{N, j}_t - X^{N, i}_t\bigr) \, \d (W^j_t - W^i_t) - \bigl(L^{N, j}_t + L^{N, i}_t\bigr) + \d \Lambda^{N, j}_t,
\end{align}
where $\Lambda^{N, j}$ is half of the local time of $X^{N, j} - X^{N, i}$ at zero and $\sgn(x) = 1$ if $x > 0$ and $\sgn(x) = -1$ if $x \leq 0$. We used in the second equality that a.s.\@ outside a nullset of $[0, T]$, $X^{N, j}_t \leq X^{N, i}_t$ whenever $L^{N, j}_t$ increases and $X^{N, i}_t < X^{N, j}_t$ whenever $L^{N, j}_t$ increases. Now, let us define the independent Brownian motions $W^{N, j} = \int_0^{\cdot} \sgn(X^{N, j}_t - X^{N, i}_t) \, \d W^j_t$ and $B^{N, j} = \int_0^{\cdot} \sgn(X^{N, j}_t - X^{N, i}_t) \, \d W^i_t$, so that \eqref{eq:reflected_difference_comp} takes the form of the SDE
\begin{equation}
    \d Z^{N, j}_t = \d \bigl(W^{N, j}_t - B^{N, j}_t\bigr) - \d \bigl(L^{N, j}_t + L^{N, i}_t\bigr) + \d \Lambda^{N, j}_t
\end{equation}
with reflection at the origin. The process $L^{N, j} + L^{N, i}$ is viewed as an exogenous input to this SDE. By standard comparison results for reflected SDEs, the process $Z^{N, j}$ is dominated by the solution $\tilde{Z}^{N, j}$ to the SDE
\begin{equation*}
    \d \tilde{Z}^{N, j}_t = \d \bigl(W^{N, j}_t - B^{N, j}_t\bigr) + \d \tilde{\Lambda}^{N, j}_t
\end{equation*}
with reflection at the origin and initial condition $\tilde{Z}^{N, j}_0 = Z^{N, j}_0 = \lvert\xi_j - \xi_i\rvert$. Consequently, setting $\tilde{\tau}^N_j = \inf\{t > 0 \define \tilde{Z}^{N, j}_t \geq \eta\}$, we have that $\tilde{\tau}^N_j \leq \tau^N_j$. From this we deduce
\begin{equation} \label{eq:non_tilde_to_tilde}
    \pr\biggl(\frac{\lvert I^N_T\rvert}{N} < \delta\biggr) \leq \pr\biggl(\frac{\lvert \tilde{I}^N_T\rvert}{N} < \delta\biggr) \leq \pr\biggl(\frac{\lvert \tilde{I}^N_T\rvert}{N} \leq \delta \biggr),
\end{equation}
where $\tilde{I}^N_T = \{j \in [N] \define \tilde{\tau}^N_j > T\}$. We will suppose for simplicity that the limit superior of the sequence $\pr(\lvert \tilde{I}^N_T\rvert/N \leq \delta)$ is achieved along the whole sequence. 

Define $\mu^N = \frac{1}{N} \sum_{j = 1}^N \delta_{\xi_j, W^{N, j}, B^{N, j}}$. Then, the sequence $(\xi_i, W^i, \mu^N)_{N \geq i}$ is tight on
\begin{equation*}
    \Omega_0 = \R \times C([0, \infty)) \times \P\bigl(\R \times C([0, \infty))^2\bigr),
\end{equation*}
so by selecting a subsequence if necessary, we may assume the sequence converges weakly to some limit distribution $\bb{P}_0$ on $\Omega_0$. Set $\Omega_{\ast} = \R \times C([0, \infty))^2 \times \Omega_0$ and define the probability measure $\pr_{\ast}$ on $\Omega_{\ast}$ by 
\begin{equation} \label{eq:def_cond_prob}
    \pr_{\ast}(A \times B) = \int_B m(A) \, \d \pr_0(x, w, m)
\end{equation}
for Borel measurable $A \subset \R \times C([0, \infty))^2$ and $B \subset \Omega_0$. Let $(\xi, W, B, \xi_0, W^0, \mu)$ denote the canonical random element on $\Omega_{\ast}$ and, for $t \geq 0$, define $\mu_t = \pi^{\#}_t \mu$, where $\pi_t \define \R \times C([0, \infty))^2 \to \R \times C([0, \infty))^2$ is given by $(x, w, b) \mapsto (x, w_{\cdot \land t}, b_{\cdot \land t})$. Then, on $\Omega_{\ast}$, we define the $\sigma$-algebras $\cal{G}_{\ast} = \sigma(\xi_0, W^0, \mu)$ and $\F_{\ast} = \cal{G}_{\ast} \lor (\xi, W, B)$ as well as the filtrations $\bb{G}^{\ast} = (\cal{G}^{\ast}_t)_{t \geq 0}$ and $\bb{F}^{\ast} = (\F^{\ast}_t)_{t \geq 0}$ by $\cal{G}^{\ast}_t = \sigma(\xi_0, W^0_s, \mu_s \define s \in [0, t])$ and $\F^{\ast}_t = \cal{G}^{\ast}_t \lor \sigma(\xi, W_s, B_s \define s \in [0, t])$ for $t \geq 0$. By appealing to the prelimit system, it can be shown that $W$, $B$, and $W^0$ are $\bb{F}^{\ast}$-Brownian motions. Moreover, the definition of $\pr_{\ast}$ in \eqref{eq:def_cond_prob} implies that $\mu_t = \L_{\pr_{\ast}}(\xi, W_{\cdot \land t}, B_{\cdot \land t} \vert \cal{G}_{\ast})$, from which it follows that $\pr_{\ast}$-a.s.\@
\begin{equation} \label{eq:early_sigma_only}
    \mu_t = \L_{\pr_{\ast}}(\xi, W_{\cdot \land t}, B_{\cdot \land t} \vert \cal{G}^{\ast}_t).
\end{equation}

Next, note that we can write $\tilde{Z}^{N, j} = \cal{R}(\xi_j - \xi_i, W^{N, j} - B^{N, j})$ for a continuous function $\cal{R} \define \R \times C([0, \infty)) \to C([0, \infty))$, so that by the continuous mapping theorem
\begin{equation} \label{eq:conv_of_reflected_tilde}
    \frac{1}{N} \sum_{j = 1}^N \delta_{\tilde{Z}^{N, j}} \Rightarrow \L_{\pr_{\ast}}\bigl(\cal{R}(\xi - \xi_0, W - B)\big\vert \cal{G}_{\ast}\bigr)
\end{equation}
as $N \to \infty$. The process $\tilde{Z} = \cal{R}(\xi - \xi_0, W - B)$ satisfies the SDE
\begin{equation}
    \d \tilde{Z}_t = \d (W - B) + \d \tilde{\Lambda}_t
\end{equation}
with reflection at the origin and initial condition $\tilde{Z}_0 = \xi - \xi_0$. Since $W$ and $B$ are independent, we have that the map $\tilde{\tau} \define C([0, \infty)) \to [0, \infty]$, $x \mapsto \inf\{t > 0 \define x_t \geq \eta\}$ is continuous at $\L_{\pr_{\ast}}(\tilde{Z})$-a.e.\@ $x \in C([0, \infty))$. This, together with \eqref{eq:conv_of_reflected_tilde} and the fact that $\pr_{\ast}(\tilde{\tau}(\tilde{Z}) = T) = 0$ implies that
\begin{equation*}
    \frac{\lvert \tilde{I}^N_T\rvert}{N} = \frac{1}{N} \sum_{j = 1}^N \bf{1}_{\{\tilde{\tau}^N_j > T\}} \Rightarrow \pr_{\ast}\bigl(\tilde{\tau}(\tilde{Z}) > T \big\vert \cal{G}_{\ast}\bigr)
\end{equation*}
as we let $N$ to infinity. In view of \eqref{eq:non_tilde_to_tilde} and the Portmanteau theorem, we obtain that
\begin{align} \label{eq:sup_limiting_prob}
    \lim_{\delta \to 0} \limsup_{N \geq i} \pr\biggl(\frac{\lvert I^N_T\rvert}{N} < \delta\biggr) &\leq \lim_{\delta \to 0} \limsup_{N \geq i} \pr\biggl(\frac{\lvert \tilde{I}^N_T\rvert}{N} \leq \delta \biggr) \notag \\
    &\leq \lim_{\delta \to 0} \pr_{\ast}\Bigl(\pr_{\ast}\bigl(\tilde{\tau}(\tilde{Z}) > T \big\vert \cal{G}_{\ast}\bigr) \leq \delta \Bigr) \notag \\
    &= \pr_{\ast}\Bigl(\pr_{\ast}\bigl(\tilde{\tau}(\tilde{Z}) > T \big\vert \cal{G}_{\ast}\bigr) = 0 \Bigr).
\end{align}
Note that $\{\tilde{\tau}(\tilde{Z}) > T\} = \{\tilde{\tau}(\tilde{Z}_{\cdot \land T}) > T\}$ and
\begin{equation*}
    \tilde{Z}_{\cdot \land T} = \cal{R}_{\cdot \land T}(\xi - \xi_0, W - B) = \cal{R}_{\cdot \land T}\bigl(\xi - \xi_0, W_{\cdot \land T}, B_{\cdot \land T}\bigr).
\end{equation*}
Since $\xi_0$ is $\cal{G}^{\ast}_T$-measurable, this in conjunction with \eqref{eq:early_sigma_only} yields that $\pr_{\ast}(\tilde{\tau}(\tilde{Z}) > T \vert \cal{G}_{\ast}) = \pr_{\ast}(\tilde{\tau}(\tilde{Z}) > T \vert \cal{G}^{\ast}_T)$. Thus, by \eqref{eq:sup_limiting_prob} it is enough to show that $\pr_{\ast}$-a.s.\@ we have $\pr_{\ast}(\tilde{\tau}(\tilde{Z}) > T \vert \cal{G}^{\ast}_T) > 0$. We establish this via a change of measure argument.

For $n \geq 1$, let us define $\varrho_n$ to be zero on $\{\lvert \xi - \xi_0\rvert \geq \eta\}$ and
\begin{equation*}
    \varrho_n = \inf\biggl\{t > 0 \define \int_0^t \frac{1}{\lvert \eta - \tilde{Z}_t\rvert^2} \, \d t \geq n\biggr\}
\end{equation*}
on $\{\lvert \xi - \xi_0\rvert < \eta\}$. Set $\varrho = \lim_{n \to \infty} \varrho_n$. With this, we can define the process $\cal{E} = (\cal{E}_t)_{t \in [0, \varrho)}$ to be the stochastic exponential of
\begin{equation*}
    \biggl(\int_0^t -\frac{1}{\eta - \tilde{Z}_s} \, \d W_s\biggr)_{t \in [0, \varrho)}.
\end{equation*}
By Novikov's condition, the process $(\cal{E}_{t \land \varrho_n})_{t \geq 0}$ is a martingale, so
\begin{equation*}
    W^n = W + \int_0^{\cdot \land \varrho_n} \frac{1}{\eta - \tilde{Z}_t} \, \d t
\end{equation*}
is an $\bb{F}^{\ast}$-Brownian motion under the measure $\pr_n$ given by $\d \pr_n = \cal{E}_{t \land \varrho_n}\, \d \pr_{\ast}$ on $\F^{\ast}_t$ for $t \geq 0$. 
% Let us now define the process $\tilde{W} = (\tilde{W}_t)_{t \in [0, \varrho)}$ by
% \begin{equation*}
%     \tilde{W}_t = W_t + \int_0^t \frac{1}{\eta - \tilde{Z}_s} \, \d s
% \end{equation*}
% for $t \in [0, \varrho)$. Then $W^n$ and $\tilde{W}$ coincide for $t \in [0, \varrho)$ and, 
Then, under $\pr_n$, $\tilde{Z}$ satisfies the reflected SDE
\begin{equation*}
    \d \tilde{Z}_t = - \frac{1}{\eta - \tilde{Z}_t} \, \d t + \d (W^n - B) + \d \tilde{\Lambda}_t
\end{equation*}
for $t \in [0, \varrho_n)$. Next, let $\tilde{Z}^n = (\tilde{Z}^n_t)_{t \geq 0}$ solve the SDE
\begin{equation*}
    \d \tilde{Z}^n_t = - \frac{1}{\eta - \tilde{Z}^n_t} \, \d t + \d (W^n - B) + \d \tilde{\Lambda}^n_t
\end{equation*}
with reflection at the origin and initial condition $\tilde{Z}^n_0 = \eta/2$, which has a unique global-in-time solution. Let $\tilde{\varrho}_n$ be the first time $t > 0$ that $\int_0^t \lvert \eta - \tilde{Z}^n_s\rvert^{-2} \, \d s \geq n$, so that the law of $\tilde{\varrho}_n$ under $\pr_n$ converges to $\infty$ as $n \to \infty$. Furthermore, on $\{\lvert \xi - \xi_0\rvert \leq \eta/2\}$, we have that $\tilde{Z}_t \leq \tilde{Z}^n_t$ for $t \in [0, \varrho_n)$, which means that $\tilde{\tau}(\tilde{Z}) \geq \varrho_n \geq \tilde{\varrho}_n$ on this set. From this, we deduce that
\begin{equation} \label{eq:simplified_probab}
    \bigl\{\tilde{\varrho}_n > T,\, \lvert \xi - \xi_0\rvert \leq \eta/2\bigr\} \subset \{\tilde{\tau}(\tilde{Z}) > T\}.
\end{equation}
Having established this, let us return to the problem at hand of proving that $\pr_{\ast}(\tilde{\tau}(\tilde{Z}) > T \vert \cal{G}^{\ast}_T) > 0$ with probability one under $\pr_{\ast}$.

Set $A_{\ast} = \{\pr_{\ast}(\tilde{\tau}(\tilde{Z}) > T \vert \cal{G}^{\ast}_T) = 0\} \in \cal{G}^{\ast}_T$. We must show that $\pr_{\ast}(A_{\ast}) = 0$. Suppose otherwise that $\pr_{\ast}(A_{\ast}) > 0$ and define $p_x = \pr_{\ast}(\lvert \xi - x\rvert \leq \eta/2)$. Since $\xi$ and $\xi_0$ have the same distribution, it holds $\pr_{\ast}$-a.s.\@ that $p_{\xi_0} > 0$. Now, define the $\bb{G}^{\ast}$-martingale $(M_t)_{t \in [0, T]}$ by $M_t = \pr_{\ast}(A_{\ast} \vert \cal{G}^{\ast}_t)$. Using that $W$ is independent of $\bb{G}^{\ast}$, we find by It\^o's formula that
\begin{equation*}
    \bf{1}_{A_{\ast}} \cal{E}_{T \land \tau_n} = \pr_{\ast}(A_{\ast} \vert \cal{G}^{\ast}_0) + \int_0^T M_t \, \d \cal{E}_{t \land \tau_n} + \int_0^T \cal{E}_{t \land \tau_n} \, \d M_t.
\end{equation*}
Multiplying both sides by $\F^{\ast}_0$-measurable $\bf{1}_{\{\lvert \xi - \xi_0\rvert \leq \eta/2\}}$ and taking expectation yields
\begin{align*}
    \pr_n\bigl(A_{\ast} \cap \{\lvert \xi - \xi_0\rvert \leq \eta/2\}\bigr) &= \ev_{\ast}\bigl[\cal{E}_{T \land \tau_n} \bf{1}_{A_{\ast} \cap \{\lvert \xi - \xi_0\rvert \leq \eta/2\}}\bigr] \\
    &= \ev_{\ast}\bigl[\pr_{\ast}(A_{\ast} \vert \cal{G}^{\ast}_0) \bf{1}_{\{\lvert \xi - \xi_0\rvert \leq \eta/2\}}\bigr] \\
    &= \ev_{\ast}\bigl[\pr_{\ast}(A_{\ast} \vert \cal{G}^{\ast}_0) p_{\xi_0}\bigr] \\
    &= \ev_{\ast}[\bf{1}_{A_{\ast}} p_{\xi_0}],
\end{align*}
where we used in the third equality that $\xi$ is independent of $\cal{G}^{\ast}_0$ while $\xi_0$ is $\cal{G}^{\ast}_0$-measurable. Since $\pr_{\ast}(A_{\ast}) > 0$ by assumption, it follows that $\ev_{\ast}[\bf{1}_{A_{\ast}} p_{\xi_0}] > 0$. Hence, the probability of $A_{\ast} \cap \{\lvert \xi - \xi_0\rvert \leq \eta/2\}$ under $\pr_n$ is positive and independent of $n \geq 1$. But we know that $\pr_n(\tilde{\varrho}_n > T) \to 1$, so we can choose $n \geq 1$ large enough such that $\pr_n\bigl(A_{\ast} \cap \{\tilde{\varrho}_n > T,\, \lvert \xi - \xi_0\rvert \leq \eta/2\}\bigr) > 0$. Now, the probability measures $\pr_n$ and $\pr_{\ast}$ are equivalent on $\F^{\ast}_T$, so the probability of $A_{\ast} \cap \{\tilde{\varrho}_n > T,\, \lvert \xi - \xi_0\rvert \leq \eta/2\}$ is also positive under $\pr_{\ast}$. In view of \eqref{eq:simplified_probab}, this finally implies the contradiction
\begin{align*}
    0 = \ev_{\ast}\bigl[\bf{1}_{A_{\ast} }\pr_{\ast}\bigl(\tilde{\tau}(\tilde{Z}) > T \big\vert \cal{G}^{\ast}_T\bigr)\bigr] \geq \pr_{\ast}\Bigl(A_{\ast} \cap \{\tilde{\varrho}_n > T,\, \lvert \xi - \xi_0\rvert \leq \eta/2\}\Bigr) > 0.
\end{align*}

\textit{Tightness of} $(\beta^N)_{N \geq 1}$: Next, let us consider the sequence $(\beta^N)_{N \geq 1}$. By the remark immediately before the statement of the proposition, we have to show that for any $n \in \bb{N}$, the family $\frac{1}{n} \beta^N\vert_{[0, n] \times \R}$, $n \geq 1$, is tight on $\P([0, n] \times \R)$. Owing to \cite[Proposition 2.2(ii)]{sznitman_poc_1991}, this is true if the family of probability measures $\cal{B}([0, n] \times \R) \ni A \mapsto \frac{1}{n}\ev[\beta^N(A)]$, $N \geq 1$, is tight. This in turn follows if we can show that for every $\epsilon > 0$ there exists $K > 0$ such that 
\begin{equation*}
    \ev\Bigl[\beta^N\Bigl([0, n] \times [-K, K]^c\Bigr)\Bigr] = \ev\biggl[\int_0^n \bf{1}_{\{\lvert B^N_t\rvert > K\}} \, \d t\biggr] = \int_0^n \pr(\lvert B^N_t \rvert > K) \, \d t < \epsilon
\end{equation*}
for all $N \geq 1$. Let us note that
\begin{equation*}
    \lvert B^N_t\rvert = (B^N_t)_- + (B^N_t)_+ \leq (B^N_t)_- + (X^{N, 1}_t)_+ \leq (B^N_t)_- + \lvert X^{N, 1}_t\rvert,
\end{equation*}
so that
\begin{equation} \label{eq:barrier_in_two}
    \int_0^n \pr(\lvert B^N_t \rvert > K) \, \d t \leq \int_0^n \pr\bigl((B^N_t)_- > K/2\bigr) \, \d t + \int_0^n \pr\bigl(\lvert X^{N, 1}_t \rvert > K/2\bigr) \, \d t.
\end{equation}
We shall bound the two terms on the right-hand side above separately. For the second one, we have
\begin{align*}
    \pr\bigl(\lvert X^{N, 1}_t\rvert > K/2\bigr) &\leq \pr\bigl(\lvert \xi_1\rvert > K/6\bigr) + \pr\bigl(\lvert W^1_t\rvert > K/6\bigr) + \pr\bigl(\lvert L^{N, 1}_t\rvert > K/6\bigr) \\
    &\leq \pr\bigl(\lvert \xi_1\rvert > K/6\bigr) + \pr\bigl(\lvert W^1_n\rvert > K/6\bigr) + \frac{6\ev[L^{N, 1}_n]}{K} \\
    &\leq \pr\bigl(\lvert \xi_1\rvert > K/6\bigr) + \pr\bigl(\lvert W^1_n\rvert > K/6\bigr) + \frac{6 \gamma n}{K}
\end{align*}
Clearly, by choosing $K > 0$ large enough, we can make the right-hand side above smaller than $\frac{\epsilon}{2n}$, so that
\begin{equation} \label{eq:bound_on_second}
    \int_0^n \pr\bigl(\lvert X^{N, 1}_t \rvert > K/2\bigr) \, \d t < \frac{\epsilon}{2}.
\end{equation}
To bound the first integral on the right-hand side of \eqref{eq:barrier_in_two}, we shall appeal to Proposition \ref{prop:comparison} \ref{it:comparison}. We introduce an auxiliary system whose initial conditions are given by $\tilde{\xi}_j = \xi_j \land 1$, $j \in [N]$. Let $\tilde{X}^1$,~\ldots, $\tilde{X}^N$ denote the corresponding solution of the Atlas model and define $\tilde{B}^N = (\tilde{B}^N_t)_{t \geq 0}$ by $\tilde{B}^N_t = \min_{j \in [N]} \tilde{X}^j_t$, so that $\tilde{B}^N$ solves the reflected SDE \eqref{eq:ps_reflected}. By Proposition \ref{prop:comparison} \ref{it:comparison}, we have that $\tilde{B}^N_t \leq B^N_t$, so that $(B^N_t)_- \leq (\tilde{B}^N_t)_-$. Now, It\^o's formula yields
\begin{align*}
    \frac{1}{N} \sum_{j = 1}^N \lvert \tilde{X}^j_n\rvert^2 &= \frac{1}{N} \sum_{j = 1}^N \biggl(\lvert \tilde{\xi}_j\rvert^2 + \int_0^n \Bigl(2 \gamma N \bf{1}_{\{\tilde{X}^j_t = \tilde{B}^N_t\}} \tilde{X}^j_t + 1\Bigr) \, \d t + \int_0^n 2\tilde{X}^j_t \, \d W^j_t\biggr) \\
    &\leq 1 + n + 2\gamma \int_0^n \tilde{B}^N_t \, \d t + \frac{1}{N} \sum_{j = 1}^N\int_0^n 2\tilde{X}^j_t \, \d W^j_t.
\end{align*}
We take expectations on both sides of the above and use that $(\tilde{B}^N_t)_- = -\tilde{B}^N_t + (\tilde{B}^N_t)_+ \leq -\tilde{B}^N_t + (\tilde{X}^1_t)_+$ and $\frac{1}{N} \sum_{j = 1}^N \lvert \tilde{X}^j_n\rvert^2 \geq 0$ to arrive at
\begin{align} \label{eq:bound_on_negative}
    2\gamma \ev\biggl[\int_0^n (B^N_t)_- \, \d t\biggr] &\leq 2\gamma \ev\biggl[\int_0^n (\tilde{B}^N_t)_- \, \d t\biggr] \notag \\
    &\leq 1 + n + \ev\biggl[\int_0^n (\tilde{X}^1_t)_+ \, \d t\biggl] \notag \\
    &\leq 1 + 2n + \frac{2\sqrt{2}}{3\sqrt{\pi}} n^{3/2} + \gamma \frac{n^2}{2}.
\end{align}
But since
\begin{equation*}
    \int_0^n \pr\bigl((B^N_t)_- > K/2\bigr) \, \d t \leq \frac{2}{K} \ev\biggl[\int_0^n (B^N_t)_- \, \d t\biggr],
\end{equation*}
the above implies that by enlarging $K$ if necessary, we have $\int_0^n \pr\bigl((B^N_t)_- > K/2\bigr) \, \d t < \frac{\epsilon}{2}$. Together with \eqref{eq:bound_on_second}, we get $\int_0^n \pr(\lvert B^N_t \rvert > K) \, \d t < \epsilon$ as required. 

\textit{Moment bound for $(L^{N, i})_{N \geq i}$:} To prove the final statement \eqref{eq:regulator_square_int}, we combine Proposition \ref{prop:comparison} \ref{it:comparison_local} and \ref{it:comparison_number}. Define the indicators $I_j = \bf{1}_{\{\xi_j \leq 1\}}$ for $j \in [N]$. By Proposition \ref{prop:comparison} \ref{it:comparison_local}, it holds that
\begin{equation*}
    \ev\bigl[\lvert L^{N, j}_t\rvert^2 \big\vert I_j = 1\bigr] \geq \ev\bigl[\lvert L^{N, j}_t\rvert^2 \big\vert I_j = 0\bigr],
\end{equation*}
which in view of the law of total probability implies that
\begin{equation} \label{eq:latp_est}
    \ev\bigl[\lvert L^{N, j}_t\rvert^2\bigr] \leq \ev\bigl[\lvert L^{N, j}_t\rvert^2 \big\vert I_j = 1\bigr].
\end{equation}
Next, consider the system consisting of the particles in $\cal{I}_N = \{j \in [N] \define I_j = 1\}$. On $\{\cal{I}_N \neq \varnothing\}$, we can solve SDE \eqref{eq:ps_reflected} with initial conditions $\xi_j$, $j \in \cal{I}_N$. Denote the solution by $\tilde{B}^N$ and let $(\tilde{X}^{N, j} - \tilde{B}^N, \tilde{L}^{N, j})$, $j \in \cal{I}_N$, be the solution to the Skorokhod problem for $\xi_j + W^j - \tilde{B}^N$. Note that these processes should be distinguished from the ones we constructed when proving tightness of $(\beta^N)_{N \geq 1}$.

Owing to Proposition \ref{prop:comparison} \ref{it:comparison_number}, on $\{\cal{I}_N \neq \varnothing\}$, we have $\tilde{L}^{N, j}_t \geq L^{N, j}_t$ for $j \in \cal{I}_N$. From this and \eqref{eq:latp_est}, we deduce
\begin{equation} \label{eq:square_bound_local}
      \ev\bigl[\lvert L^{N, i}_t\rvert^2\bigr] \leq \ev\bigl[\lvert L^{N, i}_t\rvert^2 \big\vert I_i = 1\bigr] \leq \ev\bigl[\lvert\tilde{L}^{N, i}_t\rvert^2 \big\rvert I_i = 1\bigr] = \frac{\ev\bigl[I_i \lvert\tilde{L}^{N, i}_t\rvert^2\bigr]}{\ev[I_i]}.
\end{equation}
Thus, we are done if we can show that $\ev[I_i \lvert\tilde{L}^{N, i}_t\rvert^2]$ is bounded uniformly over $N \geq i$. For $j \in \cal{I}_N$ and $t \geq 0$, let us write
\begin{align*}
    \lvert \tilde{L}^{N, j}_t\rvert^2 &= 2 \int_0^t \tilde{L}^{N, j}_s \, \d \tilde{L}^{N, j}_s \\
    &= 2 \int_0^t \tilde{X}^{N, j}_s \, \d \tilde{L}^{N, j}_s - 2 \int_0^t (\xi_j + W^j_s) \, \d \tilde{L}^{N, j}_s \\
    &\leq 2 \int_0^t \tilde{B}^N_s \, \d \tilde{L}^{N, j}_s - 2 W^j_t \tilde{L}^{N, j}_t + 2\int_0^t \tilde{L}^{N, j}_s \, \d W^j_s \\
    &\leq 2 \int_0^t \tilde{B}^N_s \, \d \tilde{L}^{N, j}_s + 2 \lvert W^j_t\rvert^2 + \frac{1}{2} \lvert \tilde{L}^{N, j}_t\rvert^2 + 2\int_0^t \tilde{L}^{N, j}_s \, \d W^i_s.
\end{align*}
We rearrange this inequality, multiply both sides by $I_j$, sum over $j \in \cal{I}_N$, and take expectation to get
\begin{equation*}
    \ev\bigl[I_i \lvert \tilde{L}^{N, i}_t\rvert^2\bigr] = \frac{1}{N} \sum_{j = 1}^N \ev\bigl[I_j \lvert \tilde{L}^{N, j}_t\rvert^2\bigr]
    \leq 4 \gamma \ev\biggl[\bf{1}_{\{\cal{I}_N \neq \varnothing\}}\int_0^t \tilde{B}^N_s \, \d s\biggr] + 4 t,
\end{equation*}
where we used that 
\begin{equation*}
    \sum_{j = 1}^N I_j \tilde{L}^{N, j}_s = \sum_{j \in \cal{I}_N} \tilde{L}^{N, j}_s = \gamma s
\end{equation*}
for $s \geq 0$ on $\{\cal{I}_N \neq \varnothing\}$. Now, analogously to the previous step, it can be shown that the expression $\ev[\bf{1}_{\{\cal{I}_N \neq \varnothing\}} \int_0^t \tilde{B}^N_s \, \d s]$ is bounded uniformly over $N \geq i$ and so the same is true for $\ev[I_i \lvert \tilde{L}^{N, i}_t\rvert^2]$. Together with \eqref{eq:square_bound_local} this concludes the proof.
\end{proof}

\section{The Mean-Field Limit of the Atlas Model} \label{sec:mfl}

The first goal of this section is to establish the convergence of the finite Atlas model to the generalised mean-field limit introduced in Definition \ref{def:generalised_solution}. Subsequently, anticipating the existence result for the strong mean-field limit, we shall prove that the reflected McKean--Vlasov SDE describing the generalised mean-field limit exhibits pathwise uniqueness. 

\subsection{Convergence to the Generalised Mean-Field Limit}

By the tightness result, Proposition \ref{prop:tightness_ps}, from Section \ref{sec:tightness}, the particle system is tight. We will now prove Theorem \ref{thm:convergence}, which states that every limit point of the particle system is a generalised solution to McKean--Vlasov SDE \eqref{eq:mfl}.

\begin{proof}[Proof of Theorem \ref{thm:convergence}]
For $N \geq 1$, let $L^{N, i}$, $i \in [N]$, and $\beta^N$ be defined as in Section \ref{sec:tightness}. Define $\mu^N = \sum_{i = 1}^N \delta_{\xi_i, W^i, L^{N, i}}$. Since the family $(\xi_i, W^i, L^{N, i})_{i \in [N]}$ is exchangeable and the sequence $(L^{N, i}, \beta^N)_{N \geq i}$ is tight on $C([0, \infty)) \times \cal{M}_1([0, \infty) \times \R)$ by Proposition \ref{prop:tightness_ps}, the same is true for $(\mu^N, \beta^N)_{N \geq 1}$ on $\Omega_0 = \cal{M}_1([0, \infty) \times \R) \times \P(\cal{S})$, where $\cal{S} = \R \times C([0, \infty)) \times C_{\text{inc}}([0, \infty))$ and $C_{\text{inc}}([0, \infty))$ denotes the set of nondecreasing continuous functions on $[0, \infty)$. Hence, by selecting a subsequence if necessary, we may assume that $(\beta^N, \mu^N)_{N \geq 1}$ converges weakly to a distribution $\pr_0$ on $\Omega_0$. We will now set up a probability measure $\pr_{\ast}$ on $\Omega_{\ast} = \cal{S} \times \Omega_0$ in a manner similar to the one used in the proof of Proposition \ref{prop:tightness_ps}. For Borel measurable $A \subset \cal{S}$ and $B \subset \Omega_0$, we define
\begin{equation*}
    \pr_{\ast}(A \times B) = \int_B m(A) \, \d \pr_0(x, w, \ell).
\end{equation*}
We let $(\xi, W, L, \beta, \mu)$ denote the canonical random element on $\Omega_{\ast}$. Next, for $t \geq 0$, let $\pi_t \define \cal{S} \to \cal{S}$ be given by $(x, w, \ell) \mapsto (x, w_{\cdot \land t}, \ell_{\cdot \land t})$. On $\Omega_{\ast}$, we define the $\sigma$-algebras $\cal{G}_{\ast} = \sigma(\beta, \mu)$ and $\F_{\ast} = \cal{G}_{\ast} \lor \sigma(\xi, W, L)$ as well as the filtrations $\bb{G}^{\ast} = (\cal{G}^{\ast}_t)_{t \geq 0}$ and $\bb{F}^{\ast} = (\F^{\ast}_t)_{t \geq 0}$ by $\cal{G}^{\ast}_t = \sigma(\beta([0, s] \times A), \mu_s \define A \in \cal{B}(\R),\, s \in [0, t])$ and $\F^{\ast}_t = \cal{G}^{\ast}_t \lor \sigma(\xi, W_s, L_s \define s \in [0, t])$ for $t \geq 0$. Similarly to \eqref{eq:early_sigma_only}, we have
\begin{equation*}
    \mu_t = \L_{\pr_{\ast}}(\xi, W_{\cdot \land t}, L_{\cdot \land t} \vert \cal{G}^{\ast}_t)
\end{equation*}
and it holds for any $i \in \bb{N}$ that $(\xi_i, W^i, L^{N, i}, \beta^N, \mu^N)_{N \geq i}$ converges weakly to $(\xi, W, L, \beta, \mu)$ on $\Omega_{\ast}$. Lastly, set $X = \xi + W + L$. We will show that $(L, \beta)$ is a solution to McKean--Vlasov SDE \eqref{eq:mfl}. For that, we have to verify Properties \ref{it:reflection} and \ref{it:minimality} from Definition \ref{def:generalised_solution}. 

Let us begin with the Property \ref{it:reflection}, fixing $\varphi \in C_b([0, \infty) \times \R)$ and $T > 0$. Note that since the first marginal of $\beta$ is the Lebesgue measure, $\pr_{\ast}$-a.s.\@ the function $(t, x) \mapsto \bf{1}_{[0, T]}(t) \varphi(t, x)$ is $\beta$-a.e.\@ continuous. Thus, it follows from the continuous mapping theorem that
\begin{equation*}
    \gamma \int_{[0, T] \times \R} \varphi(t, x) \, \d \beta^N(t,  x) \Rightarrow \gamma \int_{[0, T] \times \R} \varphi(t, x) \, \d \beta(t,  x)
\end{equation*}
as $N \to \infty$. Now, we can rewrite the integral $\gamma \int_{[0, T] \times \R} \varphi(t, x) \, \d \beta^N(t, x)$ as
\begin{equation} \label{eq:rewrite_integral}
    \frac{1}{N} \sum_{j = 1}^N \int_0^T \varphi(t, X^{N, j}_t) \, \d L^{N, j}_t = \langle \mu^N, \Phi\rangle,
\end{equation}
where $\Phi \define \R \times C([0, \infty)) \times C_{\text{inc}}([0, \infty)) \to \R$ is the continuous map given by
\begin{equation*}
    (x, w, \ell) \mapsto \int_0^T \varphi(t, x + w_t + \ell_t) \, \d \ell_t.
\end{equation*}
By \eqref{eq:regulator_square_int}, we have that the expectation of the second moment of $L^{N, j}_T$ is bounded uniformly in $N \geq j$ for $j \in \bb{N}$. Since $\Phi$ grows at most linearly in its dependence on elements in $C_{\text{inc}}([0, \infty))$, taking the weak limit as $N \to \infty$ in \eqref{eq:rewrite_integral}, we find that the random variables $\gamma \int_{[0, T] \times \R} \varphi(t, x) \, \d \beta(t, x)$ and
\begin{equation*}
    \langle \mu, \Phi\rangle = \langle \mu_T, \Phi\rangle = \ev_{\ast}\biggl[\int_0^T \varphi(t, X_t) \, \d L_t \bigg\vert \cal{G}^{\ast}_T\biggr]
\end{equation*}
on $\Omega_{\ast}$ coincide $\pr_{\ast}$-almost surly. This gives Property \ref{it:reflection}.

Let us proceed to the second property. Fix $T > 0$ and $\varphi \in C_b([0, \infty) \times \R)$ such that $x \mapsto \varphi(t, x)$ is nondecreasing for $t \geq 0$. Next, let $\tilde{L}$ be an integrable nondecreasing continuous $\bb{F}^{\xi, W}$-adapted stochastic process with $\ev[\tilde{L}_t] = \gamma t$ for $t \geq 0$. By appealing to a routine density argument, % TODO: provide more information here, maybe add result to appendix
we may assume that $\tilde{L}_t = \int_0^t \Psi_s(\xi, W_{\cdot \land s}) \, \d s$ for $t \geq 0$, where $\Psi \define \R \times C([0, \infty)) \to C([0, \infty); [c, c^{-1}])$ is a continuous function with $\ev[\Psi_s(\xi, W_{\cdot \land s})] = \gamma$ for some $c \in (0, 1]$. Our goal is to construct corresponding processes $\tilde{L}^{N, j}$, $j \in [N]$, $N \geq 1$ for the particle system such that $\frac{1}{N} \sum_{j = 1}^N \delta_{\tilde{L}^{N, j}} \Rightarrow \L_{\pr_{\ast}}(\tilde{L})$ as $N \to \infty$ and $\frac{1}{N} \sum_{j = 1}^N \tilde{L}^{N, j}_t = \gamma t$ for $t \geq 0$. To that end, set $\tilde{\cal{S}} = \R \times C([0, \infty))$ and define $\bar{\Psi} \define \tilde{\cal{S}} \times \P(\tilde{\cal{S}}) \to C([0, \infty); [c, c^{-1}])$ by
\begin{equation*}
    \bar{\Psi}_t(x, w, m) = \gamma \biggl(\int_{\tilde{\cal{S}}} \Psi_t(\tilde{x}, \tilde{w}) \, \d m(\tilde{x}, \tilde{w})\biggr)^{-1} \Psi_t(x, w).
\end{equation*}
Then, we let $\tilde{L}^{N, j}_t = \int_0^t \bar{\Psi}_t(\xi, W^j_{\cdot \land s}, \tilde{\mu}^N_s) \, \d s$ for $t \geq 0$, where $\tilde{\mu}^N_s = \frac{1}{N} \sum_{j = 1}^N \delta_{\xi_j, W^j_{\cdot \land s}}$. It follows from the law of large numbers and the continuity of $\Psi_t$ in its arguments that $\int_{\tilde{\cal{S}}} \Psi_t(x, w) \, \d \tilde{\mu}^N_t(x, w) \to \ev[\Psi_t(\xi, W_{\cdot \land t})] = \gamma$ almost surly. From this we deduce that
\begin{align*}
    \frac{1}{N} \sum_{j = 1}^N \int_0^T \varphi(t, X^{N, j}_t) \, \d \tilde{L}^{N, j}_t &= \int_{\cal{S}} \biggl(\int_0^T \varphi(t, x + w_t + \ell_t) \bar{\Psi}_t\bigl(x, w_{\cdot \land t}, \tilde{\mu}^N_t\bigr) \, \d t\biggr) \, \d \mu^N(x, w, \ell) \\
    &\Rightarrow \int_{\cal{S}} \biggl(\int_0^T \varphi(t, x + w_t + \ell_t) \Psi_t(x, w_{\cdot \land t}) \, \d t\biggr) \, \d \mu(x, w, \ell) \\
    &= \ev_{\ast}\biggl[\int_0^T \varphi(t, X_t) \, \d \tilde{L}_t \bigg\vert \cal{G}^{\ast}_T\biggr]
\end{align*}
as $N \to \infty$. Next, using that $\frac{1}{N}\sum_{j = 1}^N \tilde{L}^{N, j}_t = \gamma t$ for $t \geq 0$ by construction, we find that
\begin{align} \label{eq:ineq_ps}
    \gamma \int_{[0, T] \times \R} \varphi(t, x) \, \d \beta^N(t,  x) &= \gamma \int_0^T \varphi(t, B^N_t) \, \d t \notag \\
    &= \frac{1}{N} \sum_{j = 1}^N \int_0^T \varphi(t, B^N_t) \, \d \tilde{L}^{N, j}_t \notag \\
    &\leq \frac{1}{N} \sum_{j = 1}^N \int_0^T \varphi(t, X^{N, j}_t) \, \d \tilde{L}^{N, j}_t,
\end{align}
where we used in the last step that $x \mapsto \varphi(t, x)$ is nondecreasing and $B^N_t \leq X^{N, j}_t$. Taking the weak limit on both sides of the above yields
\begin{equation*}
    \ev_{\ast}\biggl[\int_0^T \varphi(t, X_t) \, \d \tilde{L}_t\bigg\vert \cal{G}^{\ast}_T \biggr] \geq \gamma \int_{[0, T] \times \R} \varphi(t, x) \, \d \beta(t, x),
\end{equation*}
giving Property \ref{it:minimality}.
\end{proof}

Theorem \ref{thm:convergence} does not address whether the particle system converges, only that its limit points solve McKean--Vlasov SDE \eqref{eq:mfl} in the generalised sense. In the subsequent section, we will show that this equation has a unique solution, from which the weak convergence of the particle system follows.

\subsection{Uniqueness of Generalised Solutions of McKean--Vlasov SDE \texorpdfstring{\eqref{eq:mfl}}{(EQ)}}

We will now prove that the generalised formulation of McKean--Vlasov SDE \eqref{eq:mfl} exhibits pathwise uniqueness. In doing so, we will assume that the conclusions of Theorem \ref{thm:mfl_exist} are true. The proof of Theorem \ref{thm:mfl_exist} will be delivered in Section \ref{sec:strong_existence} below. It states that for any initial condition $\xi$ whose law has a density, which is c\`adl\`ag function of finite total variation that does not vanish at zero, McKean--Vlasov SDE \eqref{eq:mfl} has a solution in the strong sense. We denote the class of all distributions on $[0, \infty)$ with that property by $\cal{I}_{\alpha}$. With this in mind, let us deliver the proof of Theorem \ref{thm:unique_generalised}.

\begin{proof}[Proof of Theorem \ref{thm:unique_generalised}]
Let $(L, \beta)$ be a generalised solution of McKean--Vlasov SDE \eqref{eq:mfl} on a given filtered probability space $(\Omega, \F, \bb{F}, \pr)$, carrying the $\F_0$-measurable initial condition $\xi$, an $\bb{F}$-Brownian motion $W$, and a subfiltration $\bb{G}$ of $\bb{F}$ independent of $(\xi, W)$. Set $X = \xi + W + L$. Fix $\epsilon > 0$ and let $\xi_{\epsilon}$ be $\F^{\xi, W}_{\epsilon}$-measurable such that $\xi_{\epsilon} \in \cal{I}_{\alpha}$ and $\ev[\lvert \xi_{\epsilon} - \xi\rvert] < \epsilon$. The reason for allowing $\xi_{\epsilon}$ to be $\F^{\xi, W}_{\epsilon}$-measurable as opposed to $\F^{\xi, W}_0$-measurable is that $\F^{\xi, W}_0$ may not include a sufficient amount of randomness to support members of $\cal{I}_{\alpha}$ at all. Given that the law of $\xi_{\epsilon}$ is in $\cal{I}_{\alpha}$, McKean--Vlasov SDE  \eqref{eq:mfl} admits a solution $b^{\epsilon} = (b^{\epsilon}_t)_{t \geq \epsilon}$ in the strong sense when started from $\xi_{\epsilon}$ at time $\epsilon$. Let us define the processes $X^{\epsilon} = (X^{\epsilon}_t)_{t \geq \epsilon}$ and $L^{\epsilon} = (L^{\epsilon}_t)_{t \geq 0}$ as follows: for $t \in [0, \epsilon)$, we set $L^{\epsilon}_t = \gamma t$, while for $t \in [\epsilon, \infty)$, we let
\begin{align*}
    L^{\epsilon}_t &= \gamma\epsilon + \sup_{\epsilon \leq s \leq t} \bigl(\xi_{\epsilon} + (W_s - W_{\epsilon})\bigr)_-, \\
    X^{\epsilon}_t &= \xi_{\epsilon} + (W_t - W_{\epsilon}) + (L^{\epsilon}_t - \gamma \epsilon).
\end{align*}
Note that $L^{\epsilon}$ is an integrable nondecreasing continuous $\bb{F}^{\xi, W}$-adapted stochastic process with $\ev[L^{\epsilon}_t] = \gamma t$ for $t \geq 0$. Hence, it serves as a test process in the sense of Property \ref{it:minimality} in the definition of generalised solutions for McKean--Vlasov SDE \eqref{eq:mfl}. 

Let us now fix an arbitrary convex function $\varphi \in C^1(\R)$ with bounded derivative and $\varphi(0) = 0$. By the fundamental theorem of calculus, we have for any $t \geq \epsilon$ that
\begin{align} \label{eq:test_equation}
    \varphi(X_t - X^{\epsilon}_t) &= \varphi(X_{\epsilon} - \xi_{\epsilon}) + \int_{\epsilon}^t \partial_x \varphi(X_s - X^{\epsilon}_s) \, \d (L_s - \tilde{L}^{\epsilon}_s) \notag \\
    &\leq \varphi(X_{\epsilon} - \xi_{\epsilon}) + \int_{\epsilon}^t \partial_x \varphi(X_s - b^{\epsilon}_s) \, \d (L_s - \tilde{L}^{\epsilon}_s) \notag \\
    &= \varphi(X_{\epsilon} - \xi_{\epsilon}) + \int_0^t \partial_x \varphi\bigl(X_s - b^{\epsilon}_{s \lor \epsilon}\bigr) \, \d (L_s - \tilde{L}^{\epsilon}_s) + \lVert \partial_x \varphi\rVert_{\infty} (L_{\epsilon} + \gamma \epsilon).
\end{align}
Here we used in the second line that $\varphi$ is convex, so its derivative $\partial_x \varphi$ is a nondecreasing function, and that $\int_{\epsilon}^s (X^{\epsilon}_u - b^{\epsilon}_u) \, \d L^{\epsilon}_u = 0$ and $X^{\epsilon}_s \geq b^{\epsilon}_s$ for $s \geq \epsilon$. Note now that the map $[0, \infty) \times \R \to \R$, $(s, x) \mapsto \partial_x \varphi(x - b^{\epsilon}_{s \lor \epsilon})$ is precisely a test function in the sense of Properties \ref{it:reflection} and \ref{it:minimality}. Thus, taking expectation on both sides of \eqref{eq:test_equation} and applying Properties \ref{it:reflection} and \ref{it:minimality} yields
\begin{equation} \label{eq:bound_from_properties}
    \ev[\varphi(X_t - X^{\epsilon}_t)] \leq \ev[\varphi(X_{\epsilon} - \xi_{\epsilon})] + 2\gamma \epsilon \lVert \partial_x \varphi\rVert_{\infty},
\end{equation}
where we used that $\ev[L_{\epsilon}] = \gamma \epsilon$. Next, it holds that
\begin{equation*}
    \ev[\lvert X_{\epsilon} - \xi_{\epsilon}\rvert] \leq \ev[\lvert \xi - \xi_{\epsilon}\rvert] + \frac{\sqrt{2\epsilon}}{\sqrt{\pi}} + \gamma \epsilon \leq (1 + \gamma) \epsilon + \frac{\sqrt{2\epsilon}}{\sqrt{\pi}},
\end{equation*}
so since $\varphi$ is of linear growth, the Vitali convergence theorem implies that $\ev[\varphi(X_{\epsilon} - \xi_{\epsilon})] \to 0$ as $\epsilon \to 0$. Inserting this into \eqref{eq:bound_from_properties} shows that $\lim_{\epsilon \to 0} \ev[\varphi(X_t - X^{\epsilon}_t)] = 0$ for $t > 0$. Since $\varphi \in C^1(\R)$ was an arbitrary convex function with bounded derivative and $\varphi(0) = 0$, it follows that $X^{\epsilon}_t \to X_t$ in probability for $t > 0$. Since $X$ has continuous trajectories, this uniquely determines it and, therefore, also the process $L$. Moreover, since $X^{\epsilon}_t$ is $\F^{\xi, W}_t$-measurable, the same is true for $X_t$. In particular, the processes $X$ and $L$ are $\bb{F}^{\xi, W}$-adapted. Since $\bb{F}^{\xi, W}$ is independent of $\bb{G}$, we deduce from Property \ref{it:reflection} that for any $T > 0$ and $\varphi \in C_b([0, \infty) \times \R)$, we have
\begin{equation*}
    \int_{[0, T] \times \R} \varphi(t, x) \, \d \beta(t, x) = \ev\biggl[\int_0^T \varphi(t, X_t) \, \d L_t \bigg\vert \cal{G}_T \biggr] = \ev\biggl[\int_0^T \varphi(t, X_t) \, \d L_t\biggr].
\end{equation*}
This not only uniquely determines $\beta$ but also implies that it is deterministic, completing the proof.
\end{proof}

\section{Analysis of McKean--Vlasov SDEs \texorpdfstring{\eqref{eq:probab_repr_hitting}}{(EQ)} and \texorpdfstring{\eqref{eq:probab_repr_reflected}}{(EQ)}} \label{sec:hitting_reflected}

The objective of this section is to study McKean--Vlasov SDEs \eqref{eq:probab_repr_hitting} and \eqref{eq:probab_repr_reflected}. We show that both equations are equivalent and that they admit at least one solution. Throughout this section, we shall assume that $(\Omega, \F, \pr)$ is a probability space carrying a Brownian motion $W$. Moreover, we shall use the convention that $[0, T[ = [0, T]$ if $T \in [0, \infty)$ and $[0, T[~= [0, \infty)$ if $T = \infty$. Lastly, for a path $f \in C([0, T[)$, we set $(f)^{\ast}_t = \sup_{0 \leq s \leq t} f_s$ and $\lvert f\rvert^{\ast}_t = \sup_{0 \leq s \leq t} \lvert f_s\rvert$ for $t \in [0, T[$.

\subsection{Equivalence between McKean--Vlasov SDEs \texorpdfstring{\eqref{eq:probab_repr_hitting}}{(EQ)} and \texorpdfstring{\eqref{eq:probab_repr_reflected}}{(EQ)}}

To establish the equivalence between the two McKean--Vlasov SDEs, we need the following elementary lemma.

\begin{lemma} \label{lem:relu_to_cdf}
Let $Z$ be an integrable random variable. Then, for any $x \in \R$, we have
\begin{equation} \label{eq:relu_to_cdf}
    \ev[(x + Z)_-] = \int_x^{\infty} \pr(y + Z \leq 0) \, \d y.
\end{equation}
\end{lemma}

\begin{proof}
Let us first suppose that $Z$ is bounded from below. Then, integration by parts gives
\begin{align*}
    \ev[(x + Z)_-] &= \ev\bigl[(x - (-Z))_-\bigr] \\
    &= \int_x^{\infty} (x - y) \, \d \L(-Z)(y) \\
    &= -\Bigl((x - y) \pr(-Z > y)\Bigr)\Big\vert_x^{\infty} + \int_x^{\infty} \pr(-Z > y) \, \d y \\
    &= \int_x^{\infty} \pr(y + Z \leq 0) \, \d y,
\end{align*}
where we use the boundedness of $Z$ from below to conclude that the boundary terms in the third line vanish, and in the last equality the fact that $\pr(y + Z < 0) = \pr(y + Z \leq 0)$ for all but countably many $y \in [x, \infty)$. Next, let $Z$ be an arbitrary integrable random variable and set $Z_n = Z \lor (-n)$ for $n \geq 0$. The above shows that
\begin{equation*}
    \ev[(x + Z_n)_-] = \int_x^{\infty} \pr(y + Z_n \leq 0) \, \d y.
\end{equation*}
Clearly, we have that $\pr(y + Z_n \leq 0) \to \pr(y + Z \leq 0)$ as $n \to \infty$. Thus, applying the monotone convergence theorem on both sides of the above equality yields \eqref{eq:relu_to_cdf}.
\end{proof}

Next, let us verify that under Assumption \ref{ass:integrability}, the integrals on the right-hand side of Equations \eqref{eq:hitting_sol_prop} and \eqref{eq:probab_repr_reflected_sol_prop} are finite. For convenience, we repeat Assumption \ref{ass:integrability}, which states that the initial condition $m$ of McKean--Vlasov SDE \eqref{eq:probab_repr_reflected} is a locally finite signed measure on $[0, \infty)$ such that $m(\{0\}) < \frac{1}{\alpha}$, $m([0, x]) \leq \frac{1}{\alpha}$ for $x \geq 0$, and for all $c > 0$ it holds that
\begin{equation*}
    \int_{[0, \infty)} e^{-c x^2} \, \d \lvert m\rvert(x) < \infty.
\end{equation*}
Furthermore, recall that the signed measure $v$ (identified with its density) appearing in \eqref{eq:hitting_sol_prop}, which serves as an initial condition for McKean--Vlasov SDE \eqref{eq:probab_repr_hitting}, is given in terms of $m$ by $v(x) = m([0, x])$ for $x \geq 0$.

\begin{lemma} \label{lem:integrals_finite}
Let Assumption \ref{ass:integrability} be satisfied. Then, for any $f \in C([0, \infty))$ and $t \geq 0$, the integrals
\begin{equation} \label{eq:hitting_and_reflecting_int}
    \int_0^{\infty} \pr\biggl(\inf_{0 \leq s \leq t} (x + W_t - f_s) \leq 0\biggr) v(x) \, \d x \quad \text{and} \quad \int_{[0, \infty)} \ev\biggl[\sup_{0 \leq s \leq t} (x + W_t - f_s)_-\biggr] \, \d m(x)
\end{equation}
are finite and coincide.
\end{lemma}

\begin{proof}
Let us start by proving that the right-hand integral in \eqref{eq:hitting_and_reflecting_int} is finite. Note that the integral is well-defined with values in $[-\infty, \infty]$, since the integrand is a bounded function and the integrator is a signed measure, so its positive part or its negative part has finite mass. Now, for $x \geq (f)^{\ast}_t$, we have with $Z = - \inf_{0 \leq s \leq t} W_s \sim \sqrt{t}\lvert W_1\rvert$ that
\begin{align} \label{eq:local_time_est}
    \ev\biggl[\sup_{0 \leq s \leq t} (x + W_s - f_s)_-\biggr] &\leq \ev\bigl[\bf{1}_{\{Z \geq x - (f)^{\ast}_t\}} Z\bigr] \notag \\
    &= \frac{\sqrt{2 t}}{\sqrt{\pi}} \int_{(x - (f)^{\ast}_t)/\sqrt{t}} y e^{-\frac{y^2}{2}} \, \d y \notag \\
    &= \frac{\sqrt{2 t}}{\sqrt{\pi}} \exp\biggl(-\frac{(x - (f)^{\ast}_t)^2}{2t}\biggr) \notag \\
    &\leq \frac{\sqrt{2 t}}{\sqrt{\pi}} \exp\biggl(\frac{3 ( (f)^{\ast}_t)^2}{2t}\biggr) e^{-\frac{x^2}{4t}}.
\end{align}
This together with the growth condition \eqref{eq:ic_growth} and the local finiteness of $m$ implies that the integral on the right-hand side of \eqref{eq:fixed_point_map} is finite.

Next, since 
\begin{equation*}
    \sup_{0 \leq s \leq t} \bigl(x + W_s - f_s\bigr)_- = \biggl(x + \inf_{0 \leq s \leq t} (W_s - f_s) \biggr)_-
\end{equation*}
for $x \geq 0$, Lemma \eqref{lem:relu_to_cdf} implies that
\begin{equation*}
    \ev\biggl[\sup_{0 \leq s \leq t} (x + W_s - f_s)_-\biggr] = \int_x^{\infty} \pr\biggl(\inf_{0 \leq s \leq t} (y + W_t - f_s) \leq 0\biggr) \, \d y
\end{equation*}
for $x \geq 0$. For notational convenience, let us set $p_y = \pr(\inf_{0 \leq s \leq t} (y + W_t - f_s) \leq 0)$. Then, integrating both sides of the above equation over $x \geq 0$ with respect to $m$ and applying Fubini's theorem to the positive and negative part of $m$ separately, implies that
\begin{align*}
    \int_{[0, \infty)} \ev\biggl[\sup_{0 \leq s \leq t} (x + W_t - f_s)_-\biggr] \, \d m(x) &= \int_{[0, \infty)} \biggl(\int_x^{\infty} p_y \, \d y \biggr) \, \d m(x) \\
    &= \int_{[0, \infty)} \biggl(\int_0^{\infty} \bf{1}_{y \geq x} p_y \, \d y \biggr) \, \d m(x) \\
    &= \int_0^{\infty} p_y \biggl(\int_{[0, \infty)} \bf{1}_{x \leq y} \, \d m(x)\biggr) \, \d y \\
    &= \int_0^{\infty} \pr\biggl(\inf_{0 \leq s \leq t} (x + W_t - f_s) \leq 0\biggr) v(x) \, \d x.
\end{align*}
In particular, the right-hand side is finite. This concludes the proof.
\end{proof}

We can now deliver the proof of the equivalence result, Proposition \ref{prop:equivalence}.

\begin{proof}[Proof of Proposition \ref{prop:equivalence}]
Let $\ell \define [0, \infty) \to \R$ be continuous and define $Y^x$ and $X^x$ as in Equations \eqref{eq:probab_repr_hitting} and \eqref{eq:probab_repr_reflected}, respectively. Additionally, let $L^x$ be the regulator associated with the process $X^x$ reflected at the origin and let $\tau_x$ be the hitting time of $Y^x$ of the origin. By Lemma \ref{lem:integrals_finite}, we have
\begin{equation*}
    \int_{[0, \infty)} \ev[L^x_t]  \, \d m(x) = \int_0^{\infty} \pr(\tau_x \leq t) v(x) \, \d x.
\end{equation*}
From this it directly follows that $\ell$ is a solution to McKean--Vlasov SDE \eqref{eq:probab_repr_hitting}, meaning that 
\begin{equation*}
    \ell_t = \int_0^{\infty} \pr(\tau_x \leq t) v(x) \, \d x
\end{equation*}
for $t \geq 0$, if and only if $\ell$ solves McKean--Vlasov SDE \eqref{eq:probab_repr_reflected}, i.e.\@
\begin{equation*}
    \ell_t = \int_{[0, \infty)} \ev[L^x_t]  \, \d m(x)
\end{equation*}
for $t \geq 0$.
\end{proof}

\subsection{Existence of McKean--Vlasov SDEs \texorpdfstring{\eqref{eq:probab_repr_hitting}}{(EQ)} and \texorpdfstring{\eqref{eq:probab_repr_reflected}}{(EQ)}}

The following lemma shows that Assumption \ref{ass:integrability} is propagated under the dynamics of McKean--Vlasov SDE \eqref{eq:probab_repr_reflected}.

\begin{lemma} \label{lem:ass_propagation}
Let Assumption \ref{ass:integrability} be satisfied. Let $f \in C([0, \infty))$, denote the solution to the Skorokhod problem for $x + W - f$ by $(X^x, L^x)$, and define the signed measure $\tilde{m}$ on $[0, \infty)$ by
\begin{equation*}
    \tilde{m}(A) = \int_{[0, \infty)} \pr(X^x_t \in A) \, \d m(x)
\end{equation*}
for Borel measurable $A \subset [0, \infty)$ and some $t \geq 0$. Then $\tilde{m}$ satisfies Assumption \ref{ass:integrability}.
\end{lemma}

\begin{proof}
Let us prove the properties stated in Assumption \ref{ass:integrability} in sequence. Local finiteness of $\tilde{m}$ is clear, so we start with $\tilde{m}(\{0\}) < \frac{1}{\alpha}$. Note that the map $[0, \infty) \ni x \mapsto \pr(X^x_t = 0)$ is nonincreasing, since if $\varrho_x = \inf\{s \in [0, t] \define X^x_s = 0\}$ occurs before or at time $t$, then $X^x_t = X^0_t$, and $X^x_t > 0$ otherwise, so that
\begin{equation*}
    \pr(X^x_t = 0) = \pr\bigl(X^x_t = 0,\, \varrho_x \leq t\bigr) = \pr\bigl(X^0_t = 0,\, \varrho_x \leq t\bigr).
\end{equation*}
Set $P(x) = \pr(X^0_t = 0) - \pr(X^x_t = 0) = \pr(X^0_t = 0,\, \varrho_x > t)$ and $\rho(x) = \frac{1}{\alpha} - m([0, x])$ for $x \geq 0$, so integration by parts yields
\begin{align*}
    \tilde{m}(\{0\}) &= \int_{[0, \infty)} \pr(X^x_t = 0) \, \d m(x) \\
    &= \int_{[0, \infty)} \pr(X^0_t = 0) \, \d m(x) - \int_{[0, \infty)} P(x) \, \d m(x) \\
    &= \pr(X^0_t = 0) m([0, \infty)) + \bigl(P(x) \rho(x) \bigr)\big\vert_0^{\infty} - \int_{[0, \infty)} \rho(x) \, \d P(x) \\
    &\leq \frac{1}{\alpha} \pr(X^0_t = 0) \\
    &< \frac{1}{\alpha},
\end{align*}
where we used in the fourth line that $\lim_{x \to \infty} \rho(x) = \frac{1}{\alpha} - m([0, \infty))$ and that $x \mapsto P(x)$ is nondecreasing while $\rho(x) \geq 0$. The last inequality follows since $\pr(X^0_t = 0) < 1$.

Let us now deduce the second property, namely $\tilde{m}([0, x]) \leq \frac{1}{\alpha}$ for $x \geq 0$. Note that $y \mapsto X^y_t$ is a.s.\@ continuous and nondecreasing, so for its generalised inverse $Z_x = \inf\{x > 0 \define X^y_t > x\}$, we have $X^y_t \leq x$ if and only if $y \leq Z_x$. Consequently, it holds for any $x \geq 0$ that
\begin{align*}
    \tilde{m}([0, x]) = \int_{[0, \infty)} \pr\bigl(X^y_t \in [0, x]\bigr) \, \d m(y) = \ev\biggl[\int_{[0, \infty)} \bf{1}_{\{y \leq Z_x\}} \, \d m(y)\biggr] = \ev\bigl[m([0, Z_x])\bigr] \leq \frac{1}{\alpha},
\end{align*}
where we applied Fubini's theorem separately to the positive and negative part of $m$ in the second equality. This gives the third property, so it remains to verify that \eqref{eq:ic_growth} holds for $\tilde{m}$. We have for $c > 0$ that
\begin{equation*}
    \int_{[0, \infty)} e^{-cx^2} \, \d \lvert \tilde{m}\rvert(x) \leq \int_{[0, \infty)} \ev\bigl[\exp(- c\lvert X^x_t\rvert^2)\bigr] \, \d \lvert m\rvert(x).
\end{equation*}
The finiteness of the integral on the right-hand side is then a consequence of the fact that
\begin{equation*}
    \lvert X^x_t\rvert^2 \geq \frac{1}{2}\lvert x\rvert^2 - \bigl\lvert W_t - f_t + L^x_t\bigr\rvert^2 \geq \frac{1}{2}\lvert x\rvert^2 - 2 \lvert W_t - f_t\rvert^2 - 2 \lvert L_t\rvert^2
\end{equation*}
and that \eqref{eq:ic_growth} is true for $m$ by assumption.
\end{proof}

We can now proceed to the proof of Theorem \ref{thm:repr_reflected_exist}. 

\begin{proof}[Proof of Theorem \ref{thm:repr_reflected_exist}]
The construction of a solution $\ell \in C([0, \infty))$ proceeds in several step. First, we employ the Schauder fixed-point theorem to establish the existence of local-in-time solution $\ell \in C([0, T])$ for some $T > 0$. Then, we use Zorn's lemma to obtain a (possibly local-in-time) solution to McKean--Vlasov SDE \eqref{eq:probab_repr_reflected} defined on a maximal time domain $[0, T_{\ast})$ for $T_{\ast} \in (0, \infty]$. Finally, arguing by contradiction, we prove that $T_{\ast} = \infty$.

\textit{Local existence:} As mentioned above, the local-in-time solution is constructed via the Schauder fixed-point theorem. To obtain a suitable subset of $C([0, T])$ (with $T > 0$ to be determined below), in which to search for the fixed point, we use \cite[Theorem 1.12]{baker_loc_times_2025}. It guarantees the existence of a solution to McKean--Vlasov SDE \eqref{eq:probab_repr_reflected} in the case where the initial condition is a finite measure. 

Let us begin with some preparation. 
We equip the space $C([0, \infty))$ with the metric
\begin{equation*}
    (f^1, f^2) \mapsto \sum_{n = 1}^{\infty} \frac{1}{2^n} \bigl(\lvert f^1 - f^2\rvert^{\ast}_n \land 1\bigr).
\end{equation*}
The space $C([0, T])$ for $T > 0$ is endowed with the supremum norm and the induced metric. We introduce the map $\Lambda \define C([0, \infty)) \to C([0, \infty))$ given by
\begin{equation} \label{eq:fixed_point_map}
    \Lambda_t(f) = \int_{[0, \infty)} \ev\biggl[\sup_{0 \leq s \leq t} (x + W_s - \alpha f_s)_-\biggr] \, \d m(x)
\end{equation}
for $t \geq 0$. Note that by Lemma \ref{lem:integrals_finite}, the integral on the right-hand side is finite for any $f \in C([0, \infty))$ and $t \geq 0$. In what follows, we will also consider the restriction of $\Lambda$ to $C([0, T])$ for $T > 0$ and, slightly abusing notation, denote this mapping by $\Lambda$ as well.

Now, let $\tilde{m}$ be any signed measure satisfying Assumption \ref{ass:integrability} such that $m([0, x]) \leq \tilde{m}([0, x])$ for any $x \geq 0$, and suppose that $\tilde{\ell}$ is a (potentially local-in-time) solution to McKean--Vlasov SDE \eqref{eq:probab_repr_reflected} defined on $[0, \tilde{T}[$ for $\tilde{T} \in (0, \infty]$. We claim that if $f \in C([0, \tilde{T}[)$ such that $f_t \leq \tilde{\ell}_t$ for $t \in [0, \tilde{T}[$, then $\Lambda_t(f) \leq \tilde{\ell}_t$ for $t \in [0, \tilde{T}[$. Indeed, applying Lemma \ref{lem:integrals_finite} twice, we see that
\begin{align} \label{eq:sol_comparison}
    \Lambda_t(f) &= \int_{[0, \infty)} \ev\biggl[\sup_{0 \leq s \leq t} (x + W_s - \alpha f_s)_-\biggr] \, \d m(x) \notag \\
    &= \int_0^{\infty} \pr\biggl(\inf_{0 \leq s \leq t} (x + W_s - \alpha f_s) \leq 0\biggr) m([0, x]) \, \d x \notag \\
    &\leq \int_0^{\infty} \pr\biggl(\inf_{0 \leq s \leq t} (x + W_s - \alpha \tilde{\ell}_s) \leq 0\biggr) \tilde{m}([0, x]) \, \d x \notag \\
    &= \int_{[0, \infty)} \ev\biggl[\sup_{0 \leq s \leq t} (x + W_s - \alpha \tilde{\ell}_s)_-\biggr] \, \d \tilde{m}(x) \notag \\
    &= \tilde{\ell}_t
\end{align}
for any $t \geq 0$, as desired. Throughout the proof, we shall frequently make use of this comparison property.

Next, we wish to define a suitable subset of the continuous functions in which to look for a fixed point. This requires further preparation. As hinted at earlier, by \cite[Theorem 1.12]{baker_loc_times_2025}, for any finite measure $\tilde{m}$ on $[0, \infty)$, McKean--Vlasov SDE \eqref{eq:probab_repr_reflected} has a unique nondecreasing continuous maximal solution $\tilde{\ell} = (\tilde{\ell}_t)_{t \in [0, \tilde{T}[}$ with initial condition $\tilde{m}$ for $\tilde{T} \in (0, \infty]$. Note that Theorem 1.12 in \cite{baker_loc_times_2025} is phrased in terms of initial conditions with unit mass. To switch back and forth between McKean--Vlasov SDE \eqref{eq:probab_repr_reflected} and the equation from \cite[Theorem 1.12]{baker_loc_times_2025}, we simply replace $\tilde{m}$ and $\tilde{\ell}$ by $\tilde{m}/\tilde{m}([0, \infty))$ and $\tilde{\ell}/\tilde{m}([0, \infty))$, respectively. Then, the coefficient appearing in front of $\ell$ in the equation from \cite{baker_loc_times_2025} becomes $\alpha \tilde{m}([0, \infty)) \in [0, \infty)$. If $\alpha \tilde{m}([0, \infty)) < 1$ or $\alpha \tilde{m}([0, \infty)) = 1$ and $\alpha \tilde{m}(\{0\}) < 1$, we are in the subcritical or critical regime, respectively, in which case \cite[Theorem 1.12]{baker_loc_times_2025} guarantees the existence of a global-in-time solution, i.e.\@ $\tilde{T} = \infty$. In the supercritical regime $\alpha \tilde{m}([0, \infty)) > 1$ and $\alpha \tilde{m}(\{0\}) < 1$, the maximal solution $\tilde{\ell}$ has a finite interval of existence, so $\tilde{T} \in (0, \infty)$.

We shall apply the above result for two distinct choices of the finite initial measure $\tilde{m}$. By the Hahn decomposition theorem, we can find mutually singular measures $m_+$ and $m_-$ on $[0, \infty)$ such that $m_+$ or $m_-$ have finite mass and $m = m_+ - m_-$. From Assumption \ref{ass:integrability}, it follows that $m_+$ must have finite mass. Indeed, if $m_-$ has finite mass, then since $m([0, x]) \leq \frac{1}{\alpha}$ for $x \geq 0$, we find that $m_+([0, \infty)) \leq \frac{1}{\alpha} + m_-([0, \infty)) < \infty$, so $m_+$ has finite mass too. Next, let us define another finite measure $\bar{m}$ on $[0, \infty)$ through its cumulative distribution function $\bar{F} \define [0, \infty) \to [0, \frac{1}{\alpha}]$, given by
\begin{equation*}
    \bar{F}(x) = \biggl(m_+([0, x]) + \frac{1}{\alpha}\bf{1}_{x \geq 1}\biggr) \land \frac{1}{\alpha}
\end{equation*}
for $x \geq 0$. It holds that $\bar{m}(\{0\}) = \bar{F}(0) < \frac{1}{\alpha}$, and $0 \leq \bar{m}([0, x]) = \bar{F}(x) \leq \frac{1}{\alpha}$ and $m([0, x]) \leq \bar{F}(x) = \tilde{m}([0, x])$ for $x \geq 0$. By the above, McKean--Vlasov SDE \eqref{eq:probab_repr_reflected} has nondecreasing solutions $\ell^+ = (\ell^+_t)_{t \in [0, T_+[}$ and $\bar{\ell} = (\bar{\ell}_t)_{t \in [0, \infty)}$ for the initial conditions $m_+$ and $\bar{m}$, where $T_+ \in (0, \infty]$. Here we let $T_+ = \infty$ and $\ell^+_t = 0$ for $t \in [0, \infty)$ if $m_+ = 0$. Note that if $m_-$ vanishes, then $m_+([0, \infty)) \leq \frac{1}{\alpha}$, so $T_+ = \infty$.   

Let us now define $\cal{K}$ to be the set of $f \in C([0, T_+[)$ such that $f_t \leq \bar{\ell}_t$ for $t \in [0, T_+[$ and $f_t - f_s \leq \ell^+_t - \ell^+_s$ for $s$, $t \in [0, T_+[$ with $s \leq t$. Note that we could have also used $\ell^+$ as an upper bound for the set $\cal{K}$. However, when we later on attempt to obtain a global solution, it is instrumental to have the globally defined upper bound $\bar{\ell}$. Clearly, $\cal{K}$ is a nonempty convex and closed subset of $C([0, T_+[)$. Next, we will prove that $\Lambda$ maps $\cal{K}$ into itself and that the image of $\cal{K}$ under $\Lambda$ is a precompact subset of $C([0, T_+[)$. Hence, if we further show that $\Lambda$ is continuous on $\cal{K}$, it follows from the Schauder fixed-point theorem that $\Lambda$ has a fixed point in $\cal{K}$. Let us begin by proving that $\Lambda(\cal{K}) \subset \cal{K}$. Fix $f \in \cal{K}$, so that $f_t \leq \bar{\ell}_t$ for $t \in [0, T_+[$ and $f_t - f_s \leq \ell^+_t - \ell^+_s$ for $s$, $t \in [0, T_+[$ with $s \leq t$. We have to show that $\Lambda_t(f) \leq \bar{\ell}_t$ for $t \in [0, T_+[$ and $\Lambda_t(f) - \Lambda_s(f) \leq \ell^+_t - \ell^+_s$ for $s$, $t \in [0, T_+[$ with $s \leq t$. The former is an immediate consequence of \eqref{eq:sol_comparison}. Next, by standard comparison results for the Skorokhod problem, it holds for any $x \geq 0$ that
\begin{equation*}
    X^{f, x}_t \geq X^{\ell^+, x}_t
\end{equation*}
for $t \in [0, T_+[$, where $(X^{g, x}, L^{g, x})$ is the solution to the Skorokhod problem for $x + W - \alpha g$ with $g \in C([0, T_+[)$. Define the measures $m^{\pm}_t$ and $\tilde{m}^{\pm}_t$, $t \in [0, T_+[$, by
\begin{equation*}
    m^{\pm}_t(A) = \int_{[0, \infty)} \pr\bigl(X^{f, x}_t \in A\bigr) \, \d m_{\pm}(x) \quad \text{and} \quad \tilde{m}^{\pm}_t(A) = \int_{[0, \infty)} \pr\bigl(X^{\ell^+, x}_t \in A\bigr) \, \d m_{\pm}(x)
\end{equation*}
for Borel measurable $A \subset [0, \infty)$ and set $m_t = m^+_t - m^-_t$. Then, for $x \geq 0$, we find $m^{\pm}_t([0, x]) \leq \tilde{m}^{\pm}_t([0, x])$ and $m_t([0, x]) \leq \tilde{m}^+_t([0, x])$. Moreover, as in the proof of Lemma \ref{lem:ass_propagation}, one can show that $m^{\pm}_t$, $\tilde{m}^{\pm}$, and $m_t$ satisfy the integrability condition \eqref{eq:ic_growth}. Note now that for $s$, $t \in [0, T_+[$ with $s \leq t$, we can write
\begin{align*}
    \Lambda_t(f) - \Lambda_s(f) &= \int_{[0, \infty)} \ev\biggl[\sup_{s \leq u \leq t} \Bigl(X^{f, x}_s + (W_u - W_s) - \alpha (f_u - f_s)\Bigr)_-\biggr] \, \d m(x) \\
    &= \int_{[0, \infty)} \ev\biggl[\sup_{s \leq u \leq t} \Bigl(x + (W_u - W_s) - \alpha (f_u - f_s)\Bigr)_-\biggr] \, \d m_s(x).
\end{align*}
Similarly, we have that 
\begin{equation*}
    \ell^+_t - \ell^+_s = \int_{[0, \infty)} \ev\biggl[\sup_{s \leq u \leq t} \Bigl(x + (W_u - W_s) - \alpha (\ell^+_u - \ell^+_s)\Bigr)_-\biggr] \, \d \tilde{m}^+_s(x).
\end{equation*}
Thus, proceeding as in \eqref{eq:sol_comparison}, using that $m_t([0, x]) \leq \tilde{m}^+_t([0, x])$ for $x \geq 0$, one can indeed show that $\Lambda_t(f) - \Lambda_s(f) \leq \ell^+_t - \ell^+_s$.

The precompactness of $\Lambda(\cal{K})$ can be argued similarly. Another estimate along the lines of \eqref{eq:sol_comparison}, using that $(m^+_t + m^-_t)([0, x]) \leq (\tilde{m}^+_t + \tilde{m}^-_t)([0, x])$ for $x \geq 0$, reveals that
\begin{equation*}
    \bigl\lvert \Lambda_t(f) - \Lambda_s(f)\bigr\rvert \leq \int_{[0, \infty)} \ev\biggl[\sup_{s \leq u \leq t} \Bigl(x + (W_u - W_s) - \alpha (\ell^+_u - \ell^+_s)\Bigr)_-\biggr] \, \d \bigl(\tilde{m}^+_s +\tilde{m}^-_s\bigr)(x)
\end{equation*}
for $s$, $t \in [0, T_+[$ with $s \leq t$. The expression on the right-hand side does not depend on $f \in C([0, T_+[)$ and converges to zero as $\lvert t - s\rvert \to 0$. Thus, the elements of $\Lambda(\cal{K})$ are equicontinuous and have the pointwise upper bound $\bar{\ell}$. This implies that $\Lambda(\cal{K})$ is a precompact subset of $C([0, T_+[)$.

To apply the Schauder fixed-point theorem, it remains to show that $\Lambda$ is continuous on $\cal{K}$. Fix a sequence $(f^k)_{k \geq 1}$ in $\cal{K}$ which converges to $f \in \cal{K}$ and let $t \in [0, T_+[$. Then, for any $k \geq 1$ and $x \geq 0$, it holds that
\begin{equation*}
     \biggl\lvert \ev\biggl[\sup_{0 \leq s \leq t} (x + W_s - \alpha f_s)_-\biggr] - \ev\biggl[\sup_{0 \leq s \leq t} (x + W_s - \alpha f^k_s)_-\biggr]\biggr\rvert \leq \alpha \lvert f - f^k\rvert^{\ast}_t.
\end{equation*}
For $x \geq \alpha \bar{\ell}_t$, the estimate \eqref{eq:local_time_est} yields the bound
\begin{equation*}
   \ev\biggl[\sup_{0 \leq s \leq t} (x + W_s - \alpha f_s)_-\biggr] + \ev\biggl[\sup_{0 \leq s \leq t} (x + W_s - \alpha f^k_s)_-\biggr] \leq \frac{2 \sqrt{2 t}}{\sqrt{\pi}} \exp\biggl(\frac{3 (\alpha \bar{\ell}_t)^2}{2t}\biggr) e^{-\frac{x^2}{4t}}.
\end{equation*}
Thus, we find
\begin{align*}
    \bigl\lvert \Lambda(f) - \Lambda(f^k)\bigr\rvert^{\ast}_t \leq \alpha \lvert m\rvert([0, z]) \lvert f - f^k\rvert^{\ast}_t + 2 \int_{(z, \infty)} \frac{\sqrt{2 t}}{\sqrt{\pi}} \exp\biggl(\frac{3 (\alpha \bar{\ell}_t)^2}{2t}\biggr) e^{-\frac{x^2}{4t}} \, \d \lvert m\rvert(x)
\end{align*}
for $z \geq \alpha \bar{\ell}_t$. Now, given any $\epsilon > 0$, we first choose $z$ sufficiently large such that the second term on the right-hand side above is smaller than $\frac{\epsilon}{2}$. Then, we let $K \geq 1$ be large enough such that for all $k \geq K$, it holds that 
\begin{equation*}
    \lvert f - f^k\rvert^{\ast}_t \leq \frac{\epsilon}{2 \alpha \lvert m\rvert([0, z])}.
\end{equation*}
Together, this implies that $\lvert \Lambda(f) - \Lambda(f^k)\rvert^{\ast}_t \leq \epsilon$ for all $k \geq K$. Since $\epsilon > 0$ war arbitrary, we obtain that $\lvert \Lambda(f) - \Lambda(f^k)\rvert^{\ast}_t \to 0$ as $k \to \infty$. This holds for any $t \in [0, T_+[$, so it finally follows that $\Lambda(f^k) \to \Lambda(f)$ as $k \to \infty$ in $C([0, T_+[)$. This concludes the first step.

\textit{Existence of maximal solutions:} Let $\cal{S}$ be the subset of $\cup_{T \in (0, \infty]} (\{T\} \times C([0, T)))$ consisting of couples $(T, \ell)$ such that $\ell = (\ell_t)_{t \in [0, T)}$ is a (possibly local-in-time) solution to McKean--Vlasov SDE \eqref{eq:probab_repr_reflected} with $\ell_t \leq \bar{\ell}_t$ for $t \in [0, T)$. Note that here we are working with half-open as opposed to closed intervals. By the previous step, the set $\cal{S}$ is nonempty. We endow $\cal{S}$ with the partial order $\leq$ given by $(T, \ell) \leq (T', \ell')$ if $T \leq T'$ and $\ell_t = \ell'_t$ for $t \in [0, T)$. We claim that every totally ordered subset of $\cal{S}$ has an upper bound in $\cal{S}$. Let $\cal{C}$ be a totally ordered subset of $\cal{S}$ and set $T_{\cal{C}} = \sup_{(T, \ell) \in \cal{C}} T$. Select a sequence $(T_n, \ell^n)_{n \geq 1}$ in $\cal{C}$ such that $\lim_{n \to \infty} T_n = T_{\cal{C}}$. We define $\ell^{\cal{C}} = (\ell^{\cal{C}}_t)_{t \in [0, T_{\cal{C}})}$ by $\ell^{\cal{C}}_t = \ell^n_t$ for $t \in [0, T_{\cal{C}})$, where $n$ is minimal with $t < T_n$. Clearly, $\ell^{\cal{C}}$ is a solution to McKean--Vlasov SDE \eqref{eq:probab_repr_reflected} on $[0, T_{\cal{C}})$ and $\ell^{\cal{C}} \leq \bar{\ell}_t$ for $t \in [0, T_{\cal{C}})$, so $(T_{\cal{C}}, \ell^\cal{C})$ is an upper bound of $\cal{C}$. This allows us to apply Zorn's lemma, yielding a maximal element $(T_{\ast}, \ell^{\ast})$ of $\cal{S}$, corresponding to a solution $\ell^{\ast}$ McKean--Vlasov SDE \eqref{eq:probab_repr_reflected} that cannot be extended to a contiguous half-open interval strictly including $[0, T_{\ast})$. 

\textit{Global existence:} It remains to prove that $T_{\ast} = \infty$. Recall that in the case that $m_-$ vanishes, the solution $\ell^+$ with initial condition $m = m_+$ is already global, so there is nothing to show. Thus, in the following we assume that $m_-$ is nontrivial. We argue by contradiction, supposing that $T_{\ast} < \infty$. Since $\ell^{\ast}$ is upper bounded by $\bar{\ell}$, the finiteness of $T_{\ast}$ implies that $\sup_{t \in [0, T_{\ast})} \ell^{\ast}_t \leq \sup_{t \in [0, T_{\ast})} \bar{\ell}_t = \bar{\ell}_{T_{\ast}} < \infty$. But then it follows from the inequalities
\begin{equation*}
    \ev\biggl[\sup_{0 \leq s \leq t} (x + W_s - \alpha \ell^{\ast}_s)_-\biggr] \leq \frac{\sqrt{2t}}{\sqrt{\pi}} + \alpha \bar{\ell}_t
\end{equation*}
for $x \geq 0$ and $t \in [0, T_{\ast})$, and
\begin{equation*}
    \ev\biggl[\sup_{0 \leq s \leq t} (x + W_s - \alpha \ell^{\ast}_s)_-\biggr] \leq \frac{\sqrt{2 t}}{\sqrt{\pi}} \exp\biggl(\frac{3 (\alpha \bar{\ell}_t)^2}{2t}\biggr) e^{-\frac{x^2}{4t}}
\end{equation*}
for $x \geq \alpha \bar{\ell}_t$ and $t \in [0, T_{\ast})$ that
\begin{equation*}
    \lVert \ell^{\ast}\rVert_{\textup{TV}, [0, t]} \leq \lvert m\rvert([0, \alpha \bar{\ell}_t]) \biggl(\frac{\sqrt{2t}}{\sqrt{\pi}} + \alpha \bar{\ell}_t\biggr) + \int_{(\alpha \bar{\ell}_t, \infty)} \frac{\sqrt{2 t}}{\sqrt{\pi}} \exp\biggl(\frac{3 (\alpha \bar{\ell}_t)^2}{2t}\biggr) e^{-\frac{x^2}{4t}} \, \d \lvert m\rvert(x),
\end{equation*}
where $\lVert \cdot\rVert_{\textup{TV}, [0, t]}$ denotes the total variation over the interval $[0, t]$. Since the right-hand side is bounded uniformly over $t \in [0, T_{\ast})$, the function $\ell^{\ast} \in C([0, T_{\ast}))$ can be extended to a bounded and continuous function on $[0, T_{\ast}]$. 

Now, define the signed measure $m_{\ast}$ on $[0, \infty)$ by
\begin{equation*}
    m_{\ast}(A) = \int_{[0, \infty)} \pr\bigl(X^{\ell^{\ast}, x}_{T_{\ast}} \in A\bigr) \, \d m(x)
\end{equation*}
for Borel measurable $A \subset [0, \infty)$. The signed measure $m_{\ast}$ satisfies Assumption \ref{ass:integrability} by Lemma \ref{lem:ass_propagation}. Hence, by the first step of the proof, we can find a (possibly local-in-time) solution $\tilde{\ell}^{\ast}$ with initial condition $m_{\ast}$ on some interval $[0, \tilde{T}_{\ast})$ for $\tilde{T}_{\ast} \in (0, \infty]$ such that $\tilde{\ell}^{\ast}_t \leq \bar{\ell}_{t - T_{\ast}}$ for $t \in [0, T_{\ast})$. But then the continuous function on $[0, T_{\ast} + \tilde{T}_{\ast})$ defined to be equal to $\ell^{\ast}_t$ for $t \in [0, T_{\ast}]$ and $\ell^{\ast}_{T_{\ast}} + \tilde{\ell}^{\ast}_{t - T_{\ast}}$ for $t \in (T_{\ast}, T_{\ast} + \tilde{T}_{\ast})$ solves McKean--Vlasov SDE \eqref{eq:probab_repr_reflected} on the larger half-open interval $[0, T_{\ast} + \tilde{T}_{\ast})$ and is upper bounded by $\bar{\ell}$. This contradicts the maximality of $\ell^{\ast}$ and, thus, concludes the proof.
\end{proof}

The final missing link is the equivalence between the mean-field limit, McKean--Vlasov SDE \eqref{eq:mfl}, and the probabilistic representation of the supercooled Stefan problem provided by McKean--Vlasov SDE \eqref{eq:probab_repr_hitting}. We shall establish this in the following section.

\section{Existence of Strong Solutions of McKean--Vlasov SDE \texorpdfstring{\eqref{eq:mfl}}{(EQ)}} \label{sec:strong_existence}

In this section, we obtain the existence of solutions to the reflected McKean--Vlasov SDE \eqref{eq:mfl} in the strong sense. This existence result follows from the equivalence between McKean--Vlasov SDEs \eqref{eq:probab_repr_hitting} and \eqref{eq:mfl}. To obtain the latter we will require a superposition principle for reflected SDEs on the half-line $[0, \infty)$ with drifts of finite variation. We will, in fact, provide a more general result that covers other cases that may be of independent interest. For simplicity, we restrict our attention to bounded coefficients.

Let $b$, $\sigma \define [0, \infty) \times [0, \infty) \to \R$ be bounded and measurable functions and let $f$, $g \in C([0, \infty))$ be nondecreasing paths started from the origin. Note that for the purpose of stating and proving the superposition principle, the symbol $b$ will be used to denote a coefficient. This should not be confused with the usage of $b$ as a solution of McKean--Vlasov SDE \eqref{eq:mfl} in the strong sense as it appears in the proof of Theorem \ref{thm:equivalence} further below. We consider the Fokker--Planck equation
\begin{equation} \label{eq:fpe_reflected}
    \d \langle \mu_t, \varphi\rangle = \bigl\langle \mu_t, b(t, \cdot) \partial_x \varphi\bigr\rangle \, \d f_t + \frac{1}{2} \bigl\langle \mu_t, \sigma^2(t, \cdot) \partial_x^2 \varphi\bigr\rangle \, \d g_t
\end{equation}
for $\varphi \in C^2_c(\R)$ with $\partial_x \varphi(0) = 0$ with an initial condition $\mu_0 \in \P([0, \infty))$. 

\begin{proposition} \label{prop:superposition}
Suppose that $\mu = (\mu_t)_{t \geq 0} \in C([0, \infty); \P([0, \infty)))$ is a solution to PDE \eqref{eq:fpe_reflected} such that $\int_0^t \mu_s(\{0\}) \, \d f_s = 0$ for all $t \geq 0$. Then, $\mu_t = \L(X_t)$ for $t \geq 0$, where $X$ is a weak solution of the reflected SDE
\begin{equation} \label{eq:sde_super}
    \d X_t = b(t, X_t) \, \d f_t + \sigma(t, X_t) \, \d W_{g_t} + \d L_t,
\end{equation}
with initial condition $X_0 \sim \mu_0$ independent of the Brownian motion $W$.
\end{proposition}

\begin{proof}
We proceed in several steps. First, we transform PDE \eqref{eq:fpe_reflected} to a PDE on the whole space. Next, we apply a suitable time change to obtain a PDE that is driven by differentiable integrators in place of $f$ and $g$. This PDE will be within the scope of the classical superposition principle by Figalli \cite[Theorem 2.6]{figalli_superposition_2008}, which does most of the heavy lifting of the proof. This superposition principle provides a weak solution to a modified SDE, from which we can construct a weak solution to SDE \eqref{eq:sde_super} by reversing the transformations applied to PDE \eqref{eq:fpe_reflected}.

Let us define the flow of probability measures $\tilde{\mu} = (\tilde{\mu}_t)_{t \geq 0}$ on $\R$ by
\begin{equation*}
    \tilde{\mu}_t(A) = \frac{1}{2} \int_{[0, \infty)} \bf{1}_A(x) + \bf{1}_A(-x) \, \d \mu_t(x)
\end{equation*}
for $A \in \cal{B}(\R)$ and $t \geq 0$. Additional, define the coefficients $\tilde{b}$, $\tilde{\sigma} \define [0, \infty) \times \R \to \R$ by $\tilde{b}(t, x) = b(t, \lvert x\rvert)$ and $\tilde{\sigma}(t, x) = \sigma(t, \lvert x\rvert)$ for $(t, x) \in [0, \infty) \times \R$. Note that the definition of $\tilde{\mu}_t$ implies that for any measurable and bounded $\varphi \define \R \to \R$ that $\langle \tilde{\mu}_t, \varphi\rangle = \frac{1}{2} \langle \mu_t, \varphi + \varphi_-\rangle$, where $\varphi_-(x) = \varphi(x)$ for $x \in \R$. Now, fix a function $\varphi \in C_c^2(\R)$ and define $\psi \in C_c^2(\R)$ by $\psi = \frac{1}{2} (\varphi + \varphi_-)$. Then, we compute
\begin{equation*}
    \partial_x \psi = \frac{1}{2}\bigl(\partial_x\varphi - (\partial_x \varphi)_-\bigr) = \frac{1}{2} \bigl(\sgn \partial_x \varphi + (\sgn \partial_x \varphi)_-\bigr) - \bf{1}_{\{x = 0\}} \partial_x \varphi
\end{equation*}
and $\partial_x^2 \psi = \frac{1}{2}(\partial_x^2 \varphi + (\partial_x^2 \varphi)_-)$, where we recall that $\sgn(x) = \bf{1}_{x > 0} - \bf{1}_{x \leq 0}$. Note that the former in particular implies that $\partial_x \psi(0) = 0$, so that $\psi$ is an admissible test function for PDE \eqref{eq:fpe_reflected}. Consequently, we deduce
\begin{align} \label{eq:fpe_whole}
    \d \langle \tilde{\mu}_t, \varphi\rangle &= \d \langle \mu_t, \psi\rangle \notag \\
    &= \bigl\langle \mu_t, b(t, \cdot) \partial_x \psi\bigr\rangle \, \d f_t + \frac{1}{2} \bigl\langle \mu_t, \sigma^2(t, \cdot) \partial_x^2 \psi\bigr\rangle \, \d g_t \notag \\
    &= \frac{1}{2} \Bigl\langle \mu_t, b(t, \cdot) \bigl(\sgn \partial_x \varphi + (\sgn \partial_x \varphi)_-\bigr)\Bigr\rangle \, \d f_t + \frac{1}{4} \Bigl\langle \mu_t, \sigma^2(t, \cdot) \bigl(\partial_x^2 \varphi + (\partial_x^2 \varphi)_-\bigr)\Bigr\rangle \, \d g_t \notag \\
    &\ \ \ - \partial_x \varphi(0) \mu_t(\{0\}) \, \d f_t \notag \\
    &= \bigl\langle \tilde{\mu}_t, \sgn \tilde{b}(t, \cdot) \partial_x \varphi\bigr\rangle \, \d f_t + \frac{1}{2} \bigl\langle \tilde{\mu}_t, \tilde{\sigma}^2(t, \cdot) \partial_x^2 \varphi\bigr\rangle \, \d g_t.
\end{align}
Here we used in the last equality that $\tilde{b}(t, \cdot)_- = \tilde{b}(t, \cdot)$ and $\tilde{\sigma}(t, \cdot)_- = \tilde{\sigma}(t, \cdot)$, and that $\int_0^t \mu_s(\{0\}) \, \d f_s$ vanishes. The above PDE \eqref{eq:fpe_whole} is the desired whole space equation for $\tilde{\mu}$. 

Next, we introduce the time change. Let us define the $\varrho \define [0, \infty) \to [0, \infty)$ to be the inverse of the continuous and strictly increasing function $[0, \infty) \ni t \mapsto f_t + g_t + t$. By definition, it holds for $0 \leq s \leq t$ that
\begin{equation*}
    \bigl\lvert f_{\varrho_t} - f_{\varrho_s} \bigr\rvert + \bigl\lvert g_{\varrho_t} - g_{\varrho_s} \bigr\rvert + \lvert \varrho_t - \varrho_s \rvert = t - s,
\end{equation*}
so both $f_{\varrho}$ and $g_{\varrho}$ are $1$-Lipschitz and, consequently, admit derivatives $F$, $G \define [0, \infty) \to [0, 1]$. Let us now define the flow of probabilities $\nu = (\nu_t)_{t \geq 0}$ on $\R$ by $\nu_t = \tilde{\mu}_{\varrho_t}$ for $t \geq 0$. Appealing to the PDE \eqref{eq:fpe_whole} satisfied by $\tilde{\mu}$, a simple application of the fundamental theorem of calculus shows that
\begin{equation*}
    \d \langle \nu_t, \varphi\rangle = \bigl\langle \nu_t, \sgn F_t \tilde{b}(\varrho_t, \cdot) \partial_x \varphi\bigr\rangle \, \d t + \frac{1}{2} \bigl\langle \nu_t, G_t \tilde{\sigma}^2(\varrho_t, \cdot) \partial_x^2 \varphi\bigr\rangle \, \d t
\end{equation*}
for $\varphi \in C_c^2(\R)$. Thus, by \cite[Theorem 2.6]{figalli_superposition_2008}, on a suitable filtered probability space $(\Omega, \F, \bb{F}, \pr)$ supporting an $\bb{F}$-Brownian motion $\tilde{W}$ and a real-valued $\F_0$-measurable random variable $\tilde{\xi}$ with distribution $\nu_0 = \tilde{\mu}_0$, we can find an $\bb{F}$-adapted weak solution $\tilde{Z} = (\tilde{Z}_t)_{t \geq 0}$ to the SDE
\begin{equation*}
    \d \tilde{Z}_t = \sgn(\tilde{Z}_t) F_t \tilde{b}(\varrho_t, \tilde{Z}_t) \, \d t + \sqrt{G_t} \tilde{\sigma}(\varrho_t, \tilde{Z}_t) \, \d \tilde{W}_t.
\end{equation*}
Next, define the process $Z = (Z_t)_{t \geq 0}$ by $Z_t = \tilde{Z}_{h_t}$ for $t \geq 0$, where we set $h = f + g + \id_{[0, \infty)}$ for notational simplicity. This process is adapted to the filtration $\bb{G} = (\cal{G}_t)_{t \geq 0}$ given by $\cal{G} = \F_{h_t}$ for $t \geq 0$. By another application of the fundamental theorem of calculus, we find that
\begin{equation} \label{eq:calc_f_int}
    \int_0^{h_t} \sgn(\tilde{Z}_s) F_s \tilde{b}(\varrho_s, \tilde{Z}_s) \, \d s = \int_0^t \sgn(Z_s) \tilde{b}(s, Z_s) F_{h_s} \, \d h_s = \int_0^t \sgn(Z_s) \tilde{b}(s, Z_s) \, \d f_s,
\end{equation}
where we used in the first equality that $\tilde{Z}_{h_s} = Z_s$ and $\varrho_{h_s} = s$ by definition of $Z$ and $\varrho$. The second equality follows from the fact that $f_{\varrho_t} = \int_0^t F_s \, \d s$, which implies $f_t = \int_0^t F_{h_s} \, \d h_s$. To get the martingale term into an appropriate form, we will have to apply the Dambis--Dubins--Schwarz twice. Firstly, by enlarging the probability space if necessary, we can find another Brownian motion $\tilde{B} = (\tilde{B}_t)_{t \geq 0}$ such that $\tilde{B}_{A_t} = \int_0^{h_t} \sqrt{G_s} \tilde{\sigma}(\varrho_s, \tilde{Z}_s) \, \d \tilde{W}_s$ for $t \geq 0$, where $A = (A_t)_{\geq 0}$ is given by
\begin{equation*}
    A_t = \int_0^{h_t} G_s \tilde{\sigma}^2(\varrho_s, \tilde{Z}_s) \, \d s = \int_0^t \tilde{\sigma}^2(s, Z_s) G_{h_s} \, \d s = \int_0^t \tilde{\sigma}^2(s, Z_s) \, \d g_s.
\end{equation*}
The above calculation is analogous to \eqref{eq:calc_f_int}. Set $\tilde{M} = \tilde{B}_A$, so that $\tilde{M}$ is a $\bb{G}$-martingale with $\langle \tilde{M}\rangle = A$. It holds that $\int_0^t \frac{1}{\tilde{\sigma}^2(s, Z_s)} \, \d \langle M\rangle_s = g_t < \infty$, so the $\bb{G}$-progressively measurable process $(\frac{1}{\tilde{\sigma}(t, Z_t)})_{t \geq 0}$ is integrable with respect to $\tilde{M}$. Thus, we can define the $\bb{G}$-martingale $M = \int_0^{\cdot} \frac{1}{\tilde{\sigma}(t, Z_t)} \, \d \tilde{M}_t$. By definition, it holds that $\langle M\rangle = g$, so by the Dambis--Dubins--Schwarz, there exists a Brownian motion $B = (B_t)_{t \geq 0}$ such that $M_t = B_{g_t}$. In particular, retracing the above steps, we find that
\begin{equation*} 
    \int_0^t \tilde{\sigma}(s, Z_s) \, \d B_{g_s} = \tilde{M}_t = \tilde{B}_{A_t} = \int_0^{h_t} \sqrt{G_s} \tilde{\sigma}(\varrho_s, \tilde{Z}_s) \, \d \tilde{W}_s.
\end{equation*}
Combining the above display with \eqref{eq:calc_f_int} shows that $Z$ satisfies the SDE
\begin{equation*}
    \d Z_t = \sgn(Z_t) \tilde{b}(t, Z_t) \, \d f_t + \tilde{\sigma}(t, Z_t) \, \d B_{g_t}.
\end{equation*}

The final step is convert the above SDE, defined on the whole space, into a reflected SDE on the half-line $[0, \infty)$. For that, we define the process $X = (X_t)_{t \geq 0}$ by $X_t = \lvert Z_t\rvert$, so that $X_0 = \lvert Z_0\rvert \sim \mu_0$. By the It\^o--Tanaka formula (cf.\@ \cite[Chapter 3, Theorem 7.1]{karatzas_bmsc_1998}), we have that
\begin{align*}
    \d X_t &= \d \lvert Z_t\rvert \\
    &= \tilde{b}(t, Z_t) \, \d f_t + \tilde{\sigma}(t, Z_t) (\sgn(Z_t) \, \d B_{g_t}) + \frac{1}{2} \, \d L^Z_t \\
    &= b(t, X_t) \, \d f_t + \sigma(t, X_t) \, \d W_{g_t} + \d L_t,
\end{align*}
where $W = (W_t)_{t \geq 0}$ is a Brownian motion constructed in a similar way as the ones above, $L^Z = (L^Z)_{t \geq 0}$ is the local time of $Z$ at zero, and $L = (L)_{t \geq 0}$ is given by $L_t = \frac{1}{2} L^Z_t$. Here we used in the third equality that by definition $\tilde{b}(t, Z_t) = b(t, \lvert Z_t\rvert) = b(t, X_t)$. The same holds for the diffusion coefficient. The process $X$ is the desired weak solution to the reflected SDE \eqref{eq:sde_super}.
\end{proof}

Let us finally deliver the proof of Theorem \ref{thm:equivalence}. Throughout this proof, we let $W$ be a Brownian motion on some probability space $(\Omega, \F, \pr)$ that, in addition, carries the initial condition $\xi$ of McKean--Vlasov SDE \eqref{eq:mfl}, which is independent of $W$. Recall that we denote the distribution of $\xi$ by $\mu_0$ and have $\nu_0 = \frac{1}{\alpha} - \mu_0$. By Assumption \ref{ass:bounded_variation}, the (signed) measures $\mu_0$ and $\nu_0$ admit a c\`adl\`ag density of locally finite total variation.

\begin{proof}[Proof of Theorem \ref{thm:equivalence}]
We begin with the much more involved `if' part of the equivalence statement. This together with the uniqueness of McKean--Vlasov SDE \eqref{eq:mfl} (Proposition \ref{prop:mfl_strong_uniqueness}) then implies the `only if' direction. The idea of the proof is as follows: let $\ell$ be a solution of McKean--Vlasov SDE \eqref{eq:probab_repr_hitting} and denote the solution to the Skorokhod problem for $x + W - \alpha \ell$ by $(X^x, L^x)$ for $x \geq 0$. Set $\tau_x = \inf\{t > 0 \define X^x_t \leq 0\}$ and, for $t \geq 0$, define the locally finite signed measure $\nu_t$ on $[0, \infty)$ by
\begin{equation} \label{eq:def_of_nu}
    \nu_t(A) = \int_0^{\infty} \pr\bigl(X^x_t \in A, \, \tau_x > t\bigr) \nu_0(x) \, \d x
\end{equation}
for $A \in \cal{B}([0, \infty))$. According to the arguments outlined in the introduction, the measure $\mu_t = \frac{1}{\alpha} - \nu_t$ should be a solution to the Fokker--Planck equation
\begin{equation} \label{eq:fpe_2}
    \d \langle \mu_t, \varphi\rangle = - \alpha \langle \mu_t, \partial_x \varphi\rangle \, \d \ell_t + \frac{1}{2} \langle \mu_t, \partial_x^2 \varphi\rangle \, \d t
\end{equation}
for test functions $\varphi \in C_c^2(\R)$ with $\partial_x \varphi(0) = 0$. This is the distributional version of PDE \eqref{eq:fpe}, modulo a translation of the domain by $- \alpha \ell_t$. The requirement $\partial_x \varphi(0) = 0$ accounts for the boundary conditions of PDE \eqref{eq:fpe}. If this is indeed the case, then the superposition principle, Proposition \ref{prop:superposition}, implies that $\mu_t = \L(\tilde{X}_t)$ for $t \geq 0$, where $(\tilde{X}, L)$ is the solution to the Skorokhod problem for $\xi + W - \alpha \ell$. Thus, to conclude that $\alpha \ell$ solves McKean--Vlasov SDE \eqref{eq:mfl} in the sense of Definition \ref{def:probab_repr_hitting}, it then remains to prove that $\ev[L_t] = \gamma t$ for $t \geq 0$. Since by the occupation time formula, we have that $\ev[L_t] = \frac{1}{2} \int_0^t \mu_s(0) \, \d s$, this follows if we can show that $\mu_s(0) = 2\gamma = \frac{1}{\alpha}$ for a.e.\@ $s \geq 0$. Note that here we are tacitly assuming that $\mu_t$ has a sufficiently regular density, which we must verify too. Let us establish all of the above facts.

We begin by proving that $\mu = (\mu_t)_{t \geq 0}$ satisfies the desired PDE \eqref{eq:fpe_2}. Fix a test function $\varphi \in C_c^2(\R)$ with $\partial_x \varphi(0) = 0$. We have that
\begin{equation*}
    \langle \mu_t, \varphi\rangle = \int_0^{\infty} \Bigl(\frac{1}{\alpha} \varphi(x) - \ev\bigl[\bf{1}_{\{\tau_x > t\}} \varphi(X^x_t)\bigr] \nu_0(x)\Bigr) \, \d x.
\end{equation*}
Let us treat the two summands in the integral on the right-hand side separately. Using that $\partial_x \varphi(0) = 0$, we get that $\int_0^{\infty} \partial_x \varphi(x) \, \d x = - \varphi(0)$ and $\int_0^{\infty} \partial_x^2 \varphi(x) \, \d x = 0$. Thus, letting $\lambda$ be the Lebesgue measure on $[0, \infty)$, we find
\begin{equation} \label{eq:eq_for_lebesgue}
    \langle \lambda, \varphi\rangle = \langle \lambda, \varphi\rangle - \alpha \int_0^t \langle \lambda, \partial_x \varphi\rangle \, \d \ell_s + \frac{1}{2} \int_0^t \langle \lambda, \partial_x^2 \varphi\rangle \, \d s - \alpha \varphi(0) \ell_t
\end{equation}
for $t \geq 0$. Next, let $\nu_t^x = \pr(X^x_t \in \cdot,\, \tau_x > t)$ for $x \geq 0$ and $t \geq 0$. Then, applying It\^o's formula to $\bf{1}_{\{\tau_x > t\}} \varphi(X^x_t)$ and taking the expectation of the resulting expression yields
\begin{equation*}
    \langle \nu^x_t, \varphi\rangle = \langle \nu^x_0, \varphi\rangle - \alpha \int_0^t \langle \nu^x_s, \partial_x \varphi\rangle \, \d \ell_s + \frac{1}{2} \int_0^t \langle \nu^x_s, \partial_x^2 \varphi\rangle \, \d s - \alpha \varphi(0) \pr(\tau_x \leq t).
\end{equation*}
Now, we multiply both sides of this equation by $\nu_0(x)$ and then integrate over $x \in [0, \infty)$. Using that $\nu_t = \int_0^{\infty} \nu^x_t \nu_0(x) \, \d x$ by \eqref{eq:def_of_nu} and the fact that
\begin{equation*}
    \ell_t = \int_0^{\infty} \pr(\tau_x \leq t) \nu_0(x) \, \d x
\end{equation*}
by the solution property of $\ell$, we get that
\begin{equation*}
    \langle \nu_t, \varphi\rangle = \langle \nu_0, \varphi\rangle - \alpha \int_0^t \langle \nu_s, \partial_x \varphi\rangle \, \d \ell_s + \frac{1}{2} \int_0^t \langle \nu_s, \partial_x^2 \varphi\rangle \, \d s - \alpha \varphi(0) \ell_t.
\end{equation*}
Finally, we multiply both sides of \eqref{eq:eq_for_lebesgue} by $\frac{1}{\alpha}$ and then subtract the above equation from this expression. Since $\mu_t = \frac{1}{\alpha} \lambda - \nu_t$ for $t \geq 0$, this shows that $\mu$ solves PDE \eqref{eq:fpe_2}.

To apply Proposition \ref{prop:superposition}, we have to verify that $\mu_t$ takes values in $\P([0, \infty))$ and that $\int_0^t \mu_s(\{0\}) \, \d \ell_s = 0$. We first show that the negative part of $\mu_t$ (as a signed measure) vanishes. This is done by proving that $\mu_t$ admits a nonnegative density, which at the same time implies that $\int_0^t \mu_s(\{0\}) \, \d \ell_s$ vanishes. Let us first find a convenient expression for the measure $\nu_t$. On $\{\tau_x > t\}$, the regulator $L^x_t$ vanishes, so for any $s \in [0, t]$, we have that
\begin{equation} \label{eq:expression_for_x}
    X^x_s = x - \alpha \ell_s + W_s = - \alpha \ell_s + W_s + L^0_s + x - L^0_s = X^0_s + x - L^0_s.
\end{equation}
Next, we claim that $\tau_x > t$ if and only if $L^0_t < x$. Suppose that $\tau_x > t$ and let $\varrho = \max\{s \in [0, t] \define X^0_s = 0\} \leq t$. Then, by \eqref{eq:expression_for_x}, we have $0 < X^x_{\varrho} = X^0_{\varrho} + x - L^0_{\varrho} = x - L^0_{\varrho}$, so that $L^0_t = L^0_{\varrho} < x$. Here we used that $X^0_{\varrho}$ vanishes and that $L^0$ does not move on $[\varrho, t]$. Conversely, if $L^0_t < x$, then for any $s \in [0, t]$, we have
\begin{equation*}
    0 \leq X^0_s = - \alpha \ell_s + W_s + L^0_s < x - \alpha \ell_s + W_s \leq X^x_s.
\end{equation*}
But this means that $\tau_x$ does not occur before or at time $t$. Combining \eqref{eq:expression_for_x} and $\{\tau_x > t\} = \{L^0_t < x\}$ allows us to write
\begin{align*}
    \langle \nu_t, \varphi\rangle &= \ev\biggl[\int_0^{\infty} \bf{1}_{\{L^0_t < x\}} \varphi\bigl(X^0_t + x - L^0_t\bigr) \nu_0(x) \, \d x\biggr] \\
    &= \ev\biggl[\int_0^{\infty} \bf{1}_{\{X^0_t < x\}} \varphi(x) \nu_0\bigl(x + \alpha \ell_t - W_t\bigr) \, \d x \biggr] \\
    &= \int_0^{\infty} \ev\Bigl[\bf{1}_{\{X^0_t \leq x\}}\nu_0\bigl(x + \alpha \ell_t - W_t\bigr) \Bigr] \varphi(x) \, \d x
\end{align*}
for bounded and measurable $\varphi \define [0, \infty) \to [0, \infty)$ with compact support. In particular, letting $\varphi = \bf{1}_{[x, x + h]}$ for $x$, $h \geq 0$, gives
\begin{equation*}
    \lim_{h \to 0} \frac{1}{h} \nu_t([x, x + h]) = \ev\Bigl[\bf{1}_{\{X^0_t \leq x\}}\nu_0\bigl(x + \alpha \ell_t - W_t\bigr) \Bigr],
\end{equation*}
where we use that $\nu_0$ is c\`adl\`ag by assumption. Consequently, $\mu_t$ admits a density (denoted by the same symbol), which takes the form
\begin{align} \label{eq:exp_for_density}
    \mu_t(x) &= \frac{1}{\alpha} - \ev\Bigl[\bf{1}_{\{X^0_t \leq x\}}\nu_0\bigl(x + \alpha \ell_t - W_t\bigr) \Bigr] \notag \\
    &= \frac{1}{\alpha} \pr(X^0_t > x) + \ev\Bigl[\bf{1}_{\{X^0_t \leq x\}}\mu_0\bigl(x + \alpha \ell_t - W_t\bigr) \Bigr].
\end{align}
The second line follows from the identity $\mu_0 = 1 - \nu_0$. Clearly, the expression on the right-hand side is nonnegative, so the same is true for $\mu_t$. Let us also highlight that it is c\`adl\`ag in $x \in [0, \infty)$, so $\mu_t$ is a c\`adl\`ag function. To show that $\mu_t([0, \infty)) = 1$ for all times $t \geq 0$, let us fix a function $\kappa \in C_c^2(\R)$ such that $\kappa(x) = 0$ for $x$ outside of $[0, 2)$, $\kappa(x) = 1$ for $x \in [-1, 1]$, and $\kappa$ is nonincreasing on $[0, \infty)$. Then, set $\kappa_n(x) = \kappa((\lvert x\rvert - n)_+)$ for $x \in \R$. Clearly, $\kappa_n \in C_c^2(\R)$, $\kappa_n(x) = 1$ for $x \in [0, n + 1]$, and $\partial_x \kappa_n$ and $\partial_x^2 \kappa_n$ vanish on $[0, \infty) \setminus (n + 1, n + 2)$. Plugging $\kappa_n$ into PDE \eqref{eq:fpe_2} yields
\begin{equation*}
    \langle \mu_t, \kappa_n\rangle = \langle \mu_0, \kappa_n\rangle - \alpha \int_0^t \langle \mu_s, \partial_x \kappa_n\rangle \, \d \ell_s + \frac{1}{2} \int_0^t \langle \mu_s, \partial_x^2 \kappa_n\rangle \, \d s.
\end{equation*}
Due to the nonnegativity of $\mu_t$ and $\mu_0$, the monotone convergence theorem implies that $\langle \mu_t, \kappa_n\rangle \to \mu_t([0, \infty))$ and $\langle \mu_0, \kappa_n\rangle \to \int_0^{\infty} \mu_0(x) \, \d x = 1$ as $n \to \infty$. Thus, to conclude that $\mu_t$ has unit mass, it suffices to show that the time-integrals on the right-hand side above vanish as $n \to \infty$. We have that
\begin{equation*}
    \lvert \langle \mu_s, \partial_x \kappa_n\rangle\rvert \leq \lVert \partial_x \kappa\rVert_{\infty} \int_{n + 1}^{n + 2} \mu_s(x) \, \d x
\end{equation*}
and a similar estimate holds for $\langle \mu_s, \partial_x^2 \kappa_n\rangle$. Hence, the time-integrals vanish by the dominated convergence theorem as soon as $\lim_{x \to \infty} \mu_s(x) = 0$ for $s > 0$. Note that the first summand on the right-hand side of the expression \eqref{eq:exp_for_density}, clearly tends to zero as $x \to \infty$. The second term is bounded from above by $\ev[\mu_0(x + \alpha \ell_s - W_s)]$, where we extend $\mu_0$ to $\R$ by setting $\mu_0(x) = \mu_0(0)$ for $x < 0$. But since $\mu_0$ integrates to one over $[0, \infty)$ and $x \mapsto \ev[\mu_0(x + \alpha \ell_s - W_s)]$ is smooth for $s > 0$, it follows that $\ev[\mu_0(x + \alpha \ell_s - W_s)] \to 0$ as $x \to \infty$. The same must be true for $\mu_s$, so $\mu_t([0, \infty)) = 1$ as required. Together with the nonnegativity of $\mu_t$, this implies $\mu_t \in \P([0, \infty))$.

The observations we made so far allow us to apply the superposition principle, Proposition \ref{prop:superposition}. Thereby, $\mu_t = \L(\tilde{X}_t)$ for $t \geq 0$, where $(\tilde{X}, L)$ is the solution to the Skorokhod problem for $\xi - \alpha \ell + W$. Thus, to show that $b = \alpha \ell$ is a solution to McKean--Vlasov SDE \eqref{eq:mfl} in the sense of Definition \ref{def:solution_strong}, it remains to prove that $\ev[L_t] = \gamma t$ for $t \geq 0$. Proceeding as in \cite[Lemma 1.18]{baker_loc_times_2025}, we first derive the classical identity
\begin{equation*}
    \ev[L_t] = \frac{1}{2} \int_0^t \mu_s(0) \, \d s
\end{equation*}
for $t \geq 0$. Let $\tilde{L}^x$ denote the local time of $\tilde{X}$ at $x \geq 0$, so that $\tilde{L}^0 = 2 L$. Note that $\tilde{L}^x$ can be chosen to be continuous in $t \geq 0$ and c\`adl\`ag in $x \geq 0$. Then, by the occupation time formula (see e.g.\@ \cite[Chapter VI, Corollary 1.6]{revuz_cmbm_1999}), for any $\epsilon > 0$, it holds that
\begin{equation*}
    \frac{1}{\epsilon}\int_0^t \bf{1}_{\{\tilde{X}_s \in [0, \epsilon]\}} \, \d s = \frac{1}{\epsilon}\int_0^{\infty}  \bf{1}_{\{x \in [0, \epsilon]\}} \, \d x.
\end{equation*}
Taking expectation on both sides and then letting $\epsilon \to 0$ gives
\begin{equation} \label{eq:occ_form}
    \int_0^t \mu_s(0) \, \d s = \lim_{\epsilon \to 0} \int_0^t \frac{1}{\epsilon} \biggl(\int_0^{\epsilon} \mu_s(x) \, \d x\biggr) \, \d s = \lim_{\epsilon \to 0} \int_0^{\epsilon} \ev[\tilde{L}^x_t] \, \d x = \ev[\tilde{L}^0_t] = 2\ev[L_t].
\end{equation}
Here we used in the first and third equality that $\mu_t$ and $x \mapsto \tilde{L}^x_t$, respectively, are c\`adl\`ag. By \eqref{eq:exp_for_density}, we have
\begin{equation*}
    \mu_s(0) = \frac{1}{\alpha} \pr(X^0_s > 0) + \ev\Bigl[\bf{1}_{\{X^0_s = 0\}} \mu_0\bigl(x + \alpha \ell_s - W_t\bigr)\Bigr].
\end{equation*}
We shall prove that $\pr(X^0_t = 0) = 0$ for a.e.\@ $s \geq 0$. From this and \eqref{eq:occ_form}, it then follows that $\ev[L_t] = \frac{1}{2} \int_0^t \mu_s(0) \, \d s = \frac{1}{2\alpha} t = \gamma t$. Similarly to \eqref{eq:occ_form}, one can see that
\begin{equation*}
    \ev[L^0_t] = \lim_{\epsilon \to 0} \frac{1}{\epsilon} \int_0^t \pr(X^0_s \in [0, \epsilon]) \, \d s \geq \lim_{\epsilon \to 0} \frac{1}{\epsilon} \int_0^t \pr(X^0_s = 0) \, \d s.
\end{equation*}
But since $\ev[L^0_t]$ is finite, this can only be true if $\pr(X^0_t = 0) = 0$ for a.e.\@ $s \geq 0$. Consequently, $b = \alpha \ell$ is a solution to McKean--Vlasov SDE \eqref{eq:mfl} in the strong sense.

Now, conversely, suppose that $b$ solves McKean--Vlasov SDE \eqref{eq:mfl}. By Theorem \ref{thm:repr_reflected_exist}, McKean--Vlasov SDE \eqref{eq:probab_repr_reflected} has a solution $\ell \in C([0, \infty))$, which in view of Proposition \ref{prop:equivalence} also solves McKean--Vlasov SDE \eqref{eq:probab_repr_hitting}. By what we just proved, this means that $\alpha \ell$ is another solution to McKean--Vlasov SDE \eqref{eq:mfl} in the strong sense. But according to Proposition \ref{prop:mfl_strong_uniqueness}, the strong version of \eqref{eq:mfl} exhibits uniqueness, so it follows that $\ell = \frac{1}{\alpha} b$. In other words, $\frac{1}{\alpha} b$ solves McKean--Vlasov SDE \eqref{eq:probab_repr_hitting}. This concludes the proof.
\end{proof}

\appendix

\section{The Skorokhod Problem} \label{sec:skorokhod_problem}

We shall briefly recall the classical Skorokhod problem in one dimension. Let $D[0, \infty)$ denote the space of c\`adl\`ag functions $[0, \infty) \to \R$.

\begin{definition} \label{def:skorokhod_problem}
We say that a tuple $(x, z) \in D[0, \infty) \times D[0, \infty)$ solves the Skorokhod problem for $f \in D[0, \infty)$ with $f_0 \geq 0$ if $z$ is a nondecreasing function started from zero such that
\begin{enumerate}[noitemsep, label = (\roman*)]
    \item $x_t = f_t + z_t \geq 0$;
    \item $\int_0^t \bf{1}_{x_s > 0} \, \d z_s = 0$
\end{enumerate}
for all $t \geq 0$.
\end{definition}

Given $f \in D[0, \infty)$ with $f_0 \geq 0$, the unique solution $(x, z)$ of the Skorokhod problem for $f$ is given by $z_t = \sup_{0 \leq s \leq t} (f_s)_-$ and $x = f + z$.

\section*{Acknowledgement}
The author would like to thank Mykhaylo Shkolnikov for an insightful discussion on the topic of the manuscript.

\sloppypar
\printbibliography

\end{document}